\documentclass{amsart}
\usepackage[paperwidth=7in, paperheight=10in, margin=.875in]{geometry}
\usepackage{amsmath}
\usepackage{amsfonts}
\usepackage{amssymb}
\usepackage{hyperref, color}

            \newtheorem{thm}{Theorem}[section]
          \newtheorem{proposition}[thm]{Proposition}

          \newtheorem{lemma}[thm]{Lemma}

          \newtheorem{definition}[thm]{Definition}

          \newtheorem{rem}[thm]{Remark}
          
          \def\R{{\mathbb R}}
          \def\d{{\partial}}
          \DeclareMathOperator{\diver}{div}
          \DeclareMathOperator{\RE}{Re}
          \DeclareMathOperator{\IM}{Im}
          \def\eps{\varepsilon}
          \newcommand{\wstar}{\overset{\ast}{\rightharpoonup}}
          \newcommand{\p}{u-u'}
          \newcommand{\Ae}{\tilde{A}^\eps}
          \newcommand{\Je}{\mathcal{J}^\eps}
          \newcommand{\g}{2(\gamma-1)}
          

\begin{document}
          \title[Nonlinear M-S and QMHD]{Nonlinear Maxwell-Schr\"odinger system and Quantum Magneto-Hydrodynamics in 3D }

\author{Paolo Antonelli}
\address{Gran Sasso Science Institute, L'Aquila}
\email{paolo.antonelli@gssi.it}

\author{Michele D'Amico}
\address{Gran Sasso Science Institute, L'Aquila}
\email{michele.damico@gssi.it}

\author{Pierangelo Marcati}
\address{Gran Sasso Science Institute, L'Aquila \and Universit\`a dell'Aquila}
\email{marcati@univaq.it}

         \maketitle

          \begin{abstract}
          Motivated by some models arising in quantum plasma dynamics, in this paper we study the Maxwell-Schr\"odinger system with a power-type nonlinearity. We show the local well-posedness in 
$H^2(\R^3)\times H^{3/2}(\R^3)$ and the global existence of finite energy weak solutions, these results are then applied to the analysis of finite energy weak solutions for Quantum Magnetohydrodynamic systems.
          \end{abstract}

\date{\today}

         \section{Introduction}\label{sect:intro}
In this paper we investigate the existence of local and global in time solutions to the following 3--D nonlinear Maxwell-Schr\"odinger system
\begin{equation}\label{eq:msc}
\left\{\begin{array}{l} i\partial_{t} u=-\frac12\Delta_A u+\phi u+\vert u\vert^{2(\gamma-1)}u\\
\Box A=\mathbb PJ(u,A)\\
\end{array}\right. 
\end{equation}
with the initial data
\begin{equation*}
u(0)=u_0,\; A(0)=A_0, \;\d_tA(0)=A_1.
\end{equation*}
Here all the physical constants are normalized to 1, $\Delta_A=(\nabla-iA)^2$ denotes the magnetic Laplacian, 
$\phi=\phi(\rho)=(-\Delta)^{-1}\rho$, with $\rho:=|u|^2$, represents the Hartree-type electrostatic potential, while the power nonlinearity describes the self-consistent interaction potential. 
$J(u,A)=\IM(\bar{u}(\nabla -iA)u)$ is the electric current density and $\mathbb P=\mathbb I-\nabla\diver\Delta^{-1}$ denotes the Leray-Helmholtz projection operator onto divergence free vector fields.
\newline
The Maxwell-Schr\"odinger system
\begin{equation}\label{eq:ms}
\left\{\begin{aligned}
&i\d_tu=-\frac12\Delta_Au+\phi u\\
&-\Delta\phi-\d_t\diver A=\rho\\
&\Box A+\nabla(\d_t\phi+\diver A)=J,\\
\end{aligned}\right.
\end{equation}
is used in the literature to describe the dynamics of a charged non-relativistic quantum particle, subject to its self-generated (classical) electro-magnetic field, see for instance \cite{Sc,Fe}. In particular  the Maxwell-Schr\"odinger system \eqref{eq:ms} can be seen as a classical approximation to the quantum field equations for an electro-dynamical non-relativistic many body system. It is well known to be invariant under the gauge transformation
\begin{equation}\label{eq:gauge}
(u, A, \phi)\mapsto (u', A', \phi')=(e^{i\lambda}u, A+\nabla\lambda, \phi-\d_t\lambda),
\end{equation}
therefore for our convenience we can decide to work in the Coulomb gauge, namely by assuming $\diver A=0$. Consequently under this gauge  the system (\ref{eq:ms}) takes the form
\begin{equation*}
\left\{\begin{aligned}
&i\d_t u=-\frac12\Delta_A u+\phi u\\
&\Box A=\mathbb PJ(u,A).
\end{aligned}\right.
\end{equation*}
It is straightforward to verify that also power-type nonlinearities of the previous form are gauge invariant. 
\\ The Maxwell-Schr\"odinger system \eqref{eq:ms} has been widely studied in the mathematical literature in the various choice of gauges, 
for instance among the first mathematical treatments we mention \cite{NT, TN}, where the authors studied the local and global well-posedness in high regularity spaces by means of the Lorentz gauge. The global existence of finite energy weak solutions has been investigated in \cite{GNS}, by using the method of vanishing viscosity. However the uniqueness and the global well-posedness of the finite energy weak solutions is not easily achievable  with this approach. In \cite{NW, NW1}, by using the semigroup associated to the magnetic Laplacian following Kato's theory \cite{Kato1, Kato2} and hence by means of a fixed point argument, the authors obtained global well-posedness with higher order Sobolev regularity.
   
More recently a global well-posedness result in the energy space has been proven  in \cite{BT} by using the analysis of a short time wave packet parametrix for the magnetic Schr\"odinger equation and the related linear, bilinear, and trilinear estimates. Therefore strong $H^1$ solutions to \eqref{eq:ms} are obtained as the unique strong limit of $H^2$ solutions. Moreover  in the same paper the authors obtained a continuous dependence on initial data in the energy space.  
The asymptotic behavior and the long-range scattering of solutions to \eqref{eq:ms} has been studied for instance in \cite{GV1, GV2, Shi} (see also the references therein). 
The global well-posedness in the space of energy for the 2D Maxwell-Schr\"odinger system in Lorentz gauge has been investigated by \cite{Wada} .
\newline
In the present paper we focus on the Cauchy problem for the Maxwell-Schr\"odinger system with a power-type nonlinearity; our interest in this problem is motivated by the possibility to develop a general theory for quantum fluids in presence of self-induced electromagnetic interacting fields. 
The related Quantum Magneto-Hydrodynamic (QMHD) systems, with a nontrivial pressure tensor,  arise in the description of quantum plasmas, for example in astrophysics, where magnetic fields and quantum effects are non negligible, see \cite{Haas, HaasB, SE, ShuEl} and the references therein. The hydrodynamic equations describing a bipolar gas of ions and electrons can be recovered from the Maxwell-Schr\"odinger system \eqref{eq:msc} by applying the Madelung transforms as done  in \cite{AM1}, where the authors studied a general class of quantum fluids in the non-magnetic case. 
We refer to the Section \ref{sect:QMHD} for a more detailed discussion concerning the connection between QMHD and the Maxwell-Schr\"odinger system \eqref{eq:msc}.
\newline
We state in the sequel the two main results of this paper.
The first one regards the local well-posedness theory for \eqref{eq:msc} in $H^2(\R^3)\times H^{\frac32}(\R^3)$. More precisely let us denote  by
\begin{equation*}
X:=\left\{(u_0, A_0, A_1)\in H^2(\R^3)\times H^{3/2}(\R^3)\times H^{1/2}(\R^3)\;\textrm{s.t.}\;\diver A_0=\diver A_1=0\right\}.
\end{equation*}.
\begin{thm}[Local wellposedness]\label{th:1}
Let $\gamma>\frac{3}{2}$. For all $(u_0,A_0,A_1)\in X$ there exists a (maximal) time $0<T_{max}\leq\infty$ and a unique (maximal) solution $(u,A)$ to \eqref{eq:msc} such that 
\begin{itemize}
\item $u\in C([0,T_{max});H^2(\mathbb{R}^3))$,  
\item $A\in C([0,T_{max});H^{\frac{3}{2}}(\mathbb{R}^3)\cap C^1([0,T_{max});H^{\frac{1}{2}}(\mathbb{R}^3))$, $\diver A=0$
\item Let $\Gamma(u_0,A_0,A_1)=(u(\cdot),A(\cdot),\partial_{t} A(\cdot)),$ then \\ $\Gamma \in C(X;C([0,T];X))$, for any $0<T<T_{max}$
\end{itemize}
The following blowup alternative holds: either $T_{max}=\infty$ or $T_{max}<\infty$ and
\[\lim_{t\uparrow T_{max}}\Vert u(t)\Vert_{H^2(\mathbb{R}^3)}=\infty,\,\,\,\lim_{t\uparrow T_{max}}\Vert A(t)\Vert_{H^{\frac{3}{2}}(\mathbb{R}^3)}=\infty,\,\,\,\lim_{t\uparrow T_{max}}\Vert\partial_{t} A(t)\Vert_{H^{\frac{1}{2}}(\mathbb{R}^3)} =\infty \,.\]  
\end{thm}
Our proof  plays on the  the construction of the evolution operator associated to the magnetic Laplacian, based on Kato's  approach \cite{Kato1, Kato2}, then we perform a fixed point argument to approximate the solutions to the Maxwell-Schr\"odinger system by the classical Picard iteration. Differently from \cite{NW1}, in our case the solutions obtained by this method cannot be extended globally in time, indeed the power-type nonlinearity does not lead to a Gronwall type inequality capable to bound the higher order norms of the solution at any time, see also \cite{Pet} for a similar problem. 
\newline
To circumvent this difficulty  we  regularize the system \eqref{eq:msc} by making use of the so-called Yosida approximations of the identity, hence we are able to get the global well-posedness for the approximating system in $H^2(\R^3)\times H^{3/2}(\R^3)$. Moreover, by using the uniform bounds  provided by the higher order energy, defined by the norm of  $X$, we prove the existence of a finite energy weak solution to \eqref{eq:msc}, in the sense defined in \cite{GNS}. This is established by the following theorem.
\begin{thm}[Global Weak Solutions]\label{th:3}
Let $1<\gamma<3$, $(u_0,A_0,A_1)\in X$, then there exists, globally in time, a finite energy weak solution $(u,A)$ to \eqref{eq:msc}, such that $u\in L^\infty(\R_+;H^1(\R^3))$, $A\in L^\infty(\R_+;H^1(\R^3))\cap W^{1, \infty}(\R_+;L^2(\R^3))$.
\end{thm}
\begin{rem}
The same results can be obtained in a straightforward way, by using the previous results on the Coulomb gauge, in any other admissible gauge. 
\end{rem}
\begin{rem}
It is possible to include a Hartree (nonlocal) nonlinear potential of the form  $(\vert\cdot\vert^{-\alpha}\ast\vert u\vert^2)u$, with $0<\alpha<3$. It can be dealt in the same fashion as for power nonlinearities.  
\end{rem}

The final goal of this paper is to develop a suitable theory for the QMHD system \eqref{eq:QMHD}. The major obstacle in this direction, which is also the major difference with respect to the usual usual QHD theory, regards the possibility to give sense to the nonlinear term related to the Lorenz force. For sake of simplicity we will consider the case without the nonlinear potential. Let us recall the definition of the macroscopic hydrodynamic variables via the so-called Madelung transformations, namely 
$${\rho}:=|u|^2 \qquad J:=\RE(\bar{u} (-i\nabla-A)u)$$  \\
From the Maxwell equations we have 
$$E=-\d_tA-\nabla\phi \qquad B=\nabla\wedge A\qquad  F_L:= \rho E+J\wedge B,$$
where $E, B, F_L, \phi$ denote the Electric field, the Magnetic field, the Lorenz force and the (scalar) electrostatic potential,  respectively.  The fields equations are supplemented by the involution of the magnetic field  and in the Coulomb gauge by the Poisson equation  (here all the physical constants are normalized to one), namely 
$$ \diver B=0, \qquad  \diver E=-\Delta \phi =\rho$$
The usual energy estimates on the  Maxwell-Schr\"odinger system  \eqref{eq:msc}, as we will see in the Section \ref{sect:QMHD},  lead  to $\frac{J}{\sqrt \rho }\in L^{\infty}_tL^{2}_x$,  $\nabla {\sqrt \rho }\in L^{\infty}_tL^{2}_x$,  $J\in L^{\infty}_tL^{ 3/2}_x$, $B \in L^{\infty}_tL^2_x$, $\nabla\wedge J \in L^{\infty}_tL^{1}_x\cap L^{\infty}_tW^{-1, 3/2}_x$. Unfortunately these bounds are not sufficient to apply the compensated compactness of Tartar \cite{Mu1, Mu2, Tar} and in particular the argument  in the Lecture 40 of \cite{Tar1}, indeed  $J\notin L^2_x$ and $B\notin L^3_x$ (the boundedness in at least one of these norms would be sufficient).
Therefore the analysis of the Lorenz force for finite energy solutions needs  still to be better understood.  In \cite{AIM},the authors investigate the weak stability of the Lorentz force by a detailed frequency analysis, in the case of incompressible dynamics (where $J\in L^2).$
In \cite{BT} the authors obtain a global well-posedness result, in the sense that finite energy strong solutions are the unique limit of $H^{2}$ regular solutions, but however these solutions do not allow to treat the Lorenz force term.
The  results of \cite{NW, NW1}, obtained without the nonlinear potential, include global well-posedness in higher order Sobolev spaces which combined with the methods of \cite{AM1,AM2} allows instead to analyze the pressureless QMHD case. \\
The additional difficulty introduced by the power nonlinearity in the  Maxwell-Schr\"odinger system  \eqref{eq:msc} in 3--D, namely a nonlinear pressure term in the QMHD system, can't be easily managed. Usually the proof of higher order well-posedness for the NLS, combines higher order energy estimates with the use of sharp Strichartz estimates. However Strichartz estimates of the same type are not, to our knowledge, available for the system \eqref{eq:msc}, while a brute force higher order energy estimate would end up in a superlinear Gronwall inequality and hence into an upper bound which blows up in finite time. \\
Our theory deals with the presence of a hydrodynamic pressure and it will provide a local well-posedness of QMHD in the higher regularity framework.

The paper is organized as follows.
In Section \ref{sect:prel} we collect some estimates which will be used afterwards
and we study the evolution operator associated to the linear magnetic Schr\"odinger equation. In Section \ref{sect:LWP} we prove Theorem \ref{th:1}. In Section \ref{sect:g_sol} we introduce an approximating system to \eqref{eq:msc} for which we show global existence of solutions and then we pass to the limit, proving Theorem \ref{th:3}. Finally, in Section \ref{sect:QMHD} we discuss about the application of our main results to the existence theory for the QMHD system.

\section{Notation and Preliminaries}\label{sect:prel}
In this Section we introduce the notation and we review some preliminary results we are going to use throughout the paper.
\newline
Let $A, B$ be two quantities, we say $A\lesssim B$ if $A\leq CB$ for some constant $C>0$.
We denote by $L^p(\R^3)$ the usual Lebesgue spaces, $H^{s, p}(\R^3)$ are the Sobolev spaces defined throught the norms $\|f\|_{H^{s, p}}:=\|(1-\Delta)^{s/2}f\|_{L^p}$. For a given reflexive Banach space $\mathcal X$ we let $C([0,T];\mathcal{X})$ (resp. $C^1([0,T];\mathcal{X}))$) denote the space of continuous (resp. differentiable) maps $[0, T]\mapsto\mathcal X$. Analogously, $L^p(0, T;\mathcal X)$ is the space of functions whose Bochner integral $\|f\|_{L^p(0, T;\mathcal X)}:=\left(\int_0^T\|f(t)\|_{\mathcal X}\,dt\right)^{1/p}$ is finite. 

\begin{lemma}[Generalized Kato-Ponce inequality]\label{lemma:leibniz}
Suppose $1<p<\infty$, $s\geq 0$, $\alpha\geq 0$, $\beta\geq 0$ and $\frac1p=\frac{1}{p_i}+\frac{1}{q_i}$ with $i=1,2$, $1<q_1\leq\infty$, $1<p_2\leq\infty$. Setting $\Lambda^s=(I-\Delta)^{\frac{s}{2}}$ we have
\begin{align*}
\Vert \Lambda^s(f_1f_2)\Vert_{L^p(\mathbb{R}^3)}&\lesssim\Vert \Lambda^{s+\alpha}(f_1)\Vert_{L^{p_1}(\mathbb{R}^3)}\Vert \Lambda^{-\alpha}(f_2)\Vert_{L^{q_1}(\mathbb{R}^3)}\\&+\Vert \Lambda^{-\beta}(f_1)\Vert_{L^{p_2}(\mathbb{R}^3)}\Vert \Lambda^{s+\beta}(f_2)\Vert_{L^{q_2}(\mathbb{R}^3)}
\end{align*}

\end{lemma}
\begin{proof}
Those estimates are generalization of Kato-Ponce commutator estimates, for a proof of this Lemma see for example Theorem $1.4$ in \cite{KG}.
\end{proof}
\begin{lemma}
Let $p, q$ be such that $1\leq q<\frac32<p\leq\infty$, then
\begin{equation}\label{eq:sogge_gen}
\|(-\Delta)^{-1}f\|_{L^\infty}\lesssim\|f\|_{L^p}^\theta\|f\|_{L^q}^{1-\theta},
\end{equation}
where $\theta\in(0, 1)$ is given by $\theta=\frac{(q'-3)p'}{3(q'-p')}$.
Furthermore, the following estimates hold
\begin{align}
\Vert(-\Delta)^{-1}(f_1f_2)f_3\Vert_{L^2(\mathbb{R}^3)}&\lesssim \Vert f_1\Vert_{L^2(\mathbb{R}^3)}\Vert f_2\Vert_{L^3(\mathbb{R}^3)}\Vert f_3\Vert_{L^3(\mathbb{R}^3)}\label{eq:inverselaplacianestimate}\\
\Vert(-\Delta)^{-1}\vert f\vert^2\Vert_{L^\infty(\mathbb{R}^3)}&\lesssim\Vert f\Vert_{L^2(\mathbb{R}^3)}^2+\Vert f\Vert_{L^\infty(\mathbb{R}^3)}^2\label{eq:soggelemma}
\end{align}
\end{lemma}
\begin{proof}
Let $R>0$, then we have
\begin{equation*}
4\pi((-\Delta)^{-1}f)(x)=\int\frac{1}{|y|}f(x-y)\,dx=\int_{|y|<R}\frac{1}{|y|}f(x-y)\,dx+\int_{|y|\geq R}\frac{1}{|y|}f(x-y)\,dx,
\end{equation*}
then by H\"older's inequality we have
\begin{equation*}
\|(-\Delta)^{-1}f\|_{L^\infty}\lesssim\left(\int_{|y|<R}|y|^{-p'}\,dy\right)^{1/p'}\|f\|_{L^p}+\left(\int_{|y|\geq R}|y|^{-q'}\,dy\right)^{1/q'}\|f\|_{L^q}.
\end{equation*}
The two integrals on the right hand side are finite by the assumptions on $p, q$. By optimizing the above inequality in $R$ we then get \eqref{eq:sogge_gen}. To prove \eqref{eq:inverselaplacianestimate} we apply H\"older and Hardy-Littlewood-Sobolev inequality to get 
\begin{equation*}
\|(-\Delta)^{-1}(f_1f_2)f_3\|_{L^2}\leq\|(-\Delta)^{-1}(f_1f_2)\|_{L^6}\|f_3\|_{L^3}\lesssim\|f_1f_2\|_{L^{6/5}}\|f_3\|_{L^3}.
\end{equation*}
Using again H\"older inequality for $\Vert f_1f_2\Vert_{L^{\frac{6}{5}}(\mathbb{R}^3)}$ we get (\ref{eq:inverselaplacianestimate}).\\
Inequality \eqref{eq:soggelemma} follows from \eqref{eq:sogge_gen} by choosing $p=\infty, q=1$ and by applying Young's inequality.
\end{proof}
Next Lemma will be useful to estimate the Hartree term in the fixed point argument in Section \ref{sect:LWP}.
\begin{lemma}
Let $u\in H^2(\R^3)$, then
\begin{equation}\label{eq:hartree_lip}
\|(-\Delta)^{-1}(|u|^2)u\|_{H^2}\lesssim\|u\|_{H^{3/4}}^2\|u\|_{H^2}.
\end{equation}
\end{lemma}
\begin{proof}
We have
\begin{equation*}
\begin{aligned}
\|(-\Delta)^{-1}(|u|^2)u\|_{H^2}\lesssim&\|(-\Delta)^{-1}(1-\Delta)(|u|^2)u\|_{L^2}+\|(-\Delta)^{-1}(|u|^2)(1-\Delta)u\|_{L^2}\\
\lesssim&\|(-\Delta)^{-1}(1-\Delta)|u|^2\|_{L^6}\|u\|_{L^3}+\|(-\Delta)^{-1}|u|^2\|_{L^\infty}\|(1-\Delta)u\|_{L^2}.
\end{aligned}
\end{equation*}
By the Hardy-Littlewood-Sobolev inequality we have
\begin{equation*}
\|(-\Delta)^{-1}(1-\Delta)|u|^2\|_{L^6}\lesssim\|(1-\Delta)|u|^2\|_{L^{6/5}}\lesssim\|u\|_{L^3}\|(1-\Delta)u\|_{L^2},
\end{equation*}
where the last inequality follows from Lemma \ref{lemma:leibniz}. On the other hand, by using \eqref{eq:sogge_gen}, with $p, q$ sufficiently close to $\frac32$, and Sobolev embedding we see that
\begin{equation*}
\|(-\Delta)^{-1}|u|^2\|_{L^\infty}\lesssim\|u\|_{H^{\frac12+\eps}}^2.
\end{equation*}
Consequently,
\begin{equation*}
\|(-\Delta)^{-1}(|u|^2)u\|_{H^2}\lesssim\|u\|_{H^{\frac12+\eps}}^2\|u\|_{H^2}.
\end{equation*}
\end{proof}

\begin{lemma} Let $A\in H^1(\mathbb{R}^3)$ and $u\in H^2(\mathbb{R}^3)$. Then the following estimates hold:
\begin{align}
\Vert (\nabla-iA)u\Vert_{H^1(\mathbb{R}^3)}&\lesssim (1+\Vert A\Vert_{H^1(\mathbb{R}^3)})\Vert u\Vert_{H^2(\mathbb{R}^3)},\label{eq:nablaminusa}\\
\Vert \mathbb PJ(u,A)\Vert_{H^\frac{1}{2}(\mathbb{R}^3)}&\lesssim \Vert u\Vert_{H^1(\mathbb{R}^3)}\Vert u\Vert_{H^2(\mathbb{R}^3)}+\Vert A\Vert_{H^1(\mathbb{R}^3)}\Vert u\Vert_{H^2(\mathbb{R}^3)}^2,\label{eq:pjestimate}\\
\|\Delta_Au\|_{L^2}&\lesssim\|u\|_{H^2}+\|A\|_{H^1}^4\|u\|_{L^2},\label{eq:magLaptoH2}\\
\|u\|_{H^2}&\lesssim\|\Delta_Au\|_{L^2}+\|A\|_{H^1}^4\|u\|_{L^2},\label{eq:H2tomagLap}\\
\|(\nabla+iA)u\|_{L^6}&\lesssim\|u\|_{H^2}+\|A\|_{H^1}^4\|u\|_{L^2}.\label{eq:covar_L6}
\end{align}
\end{lemma}
\begin{proof}
We begin with the proof of \eqref{eq:nablaminusa}.
By using Lemma \ref{lemma:leibniz} we have
\begin{align*}
\Vert (\nabla-iA)u &\Vert_{H^1(\R^3)}\leq\Vert\nabla u\Vert_{H^1(\mathbb{R}^3)}+\Vert Au\Vert_{H^1(\mathbb{R}^3)}\\
&\lesssim\Vert u\Vert_{H^2(\mathbb{R}^3)}+\Vert A\Vert_{H^1(\mathbb{R}^3)}\Vert u\Vert_{L^\infty(\mathbb{R}^3)}+\Vert A\Vert_{L^6(\mathbb{R}^3)}\Vert u\Vert_{H^{1, 3}(\R^3)}\\
&\lesssim \Vert u\Vert_{H^2(\mathbb{R}^3)}+\Vert A\Vert_{H^1(\mathbb{R}^3)}\Vert u\Vert_{H^2(\mathbb{R}^3)}+\Vert A\Vert_{H^1(\mathbb{R}^3)}\Vert u\Vert_{H^\frac{3}{2}(\mathbb{R}^3)},
\end{align*}
where in the last inequality we used the Sobolev embedding theorem. Thus \eqref{eq:nablaminusa} is proved.
We now consider \eqref{eq:pjestimate}; by Lemma \ref{lemma:leibniz},
\begin{align*}
\Vert\overline{u}\nabla u \Vert_{H^{\frac{1}{2}}(\mathbb{R}^3)}&\lesssim\Vert\overline{u}\Vert_{H^{\frac{1}{2},3}(\mathbb{R}^3)}\Vert\nabla u\Vert_{L^6(\mathbb{R}^3)}+\Vert\overline{u}\Vert_{L^6(\mathbb{R}^3)}\Vert\nabla u\Vert_{H^{\frac{1}{2},3}(\mathbb{R}^3)}\\
&\lesssim\Vert\overline{u}\Vert_{H^{\frac{1}{2},3}(\mathbb{R}^3)}\Vert\nabla u\Vert_{H^1(\mathbb{R}^3)}+\Vert\overline{u}\Vert_{H^1(\mathbb{R}^3)}\Vert\nabla u\Vert_{H^{\frac{1}{2},3}(\mathbb{R}^3)}\\
&\lesssim \Vert u\Vert_{H^1(\mathbb{R}^3)}\Vert u\Vert_{H^2(\mathbb{R}^3)}
\end{align*}
and
\begin{align*}
\Vert A\vert u\vert^2\Vert_{H^\frac{1}{2}(\mathbb{R}^3)}&\lesssim\Vert A\Vert_{H^\frac{1}{2}(\mathbb{R}^3)}\Vert u\Vert_{L^\infty(\mathbb{R}^3)}^2+\Vert A\Vert_{L^6}\Vert\vert u\vert^2\Vert_{H^{\frac{1}{2},3}(\mathbb{R}^3)}\\
&\lesssim\Vert A\Vert_{H^1(\mathbb{R}^3)}\Vert u\Vert_{H^2(\mathbb{R}^3)}^2.
\end{align*}
By adding the two estimates above we then obtain
\begin{equation*}
\|\mathbb PJ\|_{H^{1/2}}\lesssim\|J\|_{H^{1/2}}\lesssim\|u\|_{H^1}\|u\|_{H^2}+\|A\|_{H^1}\|u\|_{H^2}^2.
\end{equation*}
For \eqref{eq:magLaptoH2} we have
\begin{equation*}
\begin{aligned}
\|\Delta_Au\|_{L^2}\lesssim&\|u\|_{H^2}+\|A\cdot\nabla u\|_{L^2}+\||A|^2u\|_{L^2}\\
\lesssim&\|u\|_{H^2}+\|A\|_{H^1}\|u\|_{H^{3/2}}+\|A\|_{H^1}^2\|u\|_{H^1}\\
\lesssim&\|u\|_{H^2}+\|A\|_{H^1}\|u\|_{L^2}^{1/4}\|u\|_{H^2}^{3/4}+\|A\|_{H^1}^2\|u\|_{L^2}^{1/2}\|u\|_{H^2}^{1/2}.
\end{aligned}
\end{equation*}
By using Young's inequality we obtain \eqref{eq:magLaptoH2}. Estimate \eqref{eq:H2tomagLap} is proved in an analogous way. Finally, for \eqref{eq:covar_L6} we have
\begin{equation*}
\begin{aligned}
\|(\nabla-iA)u\|_{L^6}\lesssim&\|\nabla(\nabla-iA)u\|_{L^2}\\
\lesssim&\|\Delta_Au\|_{L^2}+\|A(\nabla-iA)\|_{L^2}\\
\lesssim&\|\Delta_Au\|_{L^2}+\|A\|_{H^1}\|u\|_{H^{3/2}}+\|A\|_{H^1}^2\|u\|_{H^1}
\end{aligned}
\end{equation*}
and proceed as for the previous estimates.
\end{proof}
Let us now state the Strichartz estimates for the wave equation we are going to use in our paper. For a proof see for example \cite{GV,Tao} and references therein.
\begin {lemma}[Strichartz estimates for the wave equation]\label{lemma:strichwave}
Let $I$ be a time interval, and let $B:I\times\mathbb{R}^3\mapsto\mathbb{C}$ be a Schwartz solution to the wave equation $\Box B=F$ with initial data $B(0)=B_0$, $\partial_{t} B(0)=B_1$. Then the following estimate holds
\begin{align*}
\Vert B\Vert_{L_t^qL_x^r(I\times\mathbb{R}^3)}&+\Vert B\Vert_{C_t\dot{H}^s_x(I\times\mathbb{R}^3)}+\Vert\partial_{t} B\Vert_{C_t\dot{H}^{s-1}_x(I\times\mathbb{R}^3)}\\
&\lesssim\Vert B_0\Vert_{\dot{H}^s(\mathbb{R}^3)}+\Vert B_1\Vert_{\dot{H}^{s-1}}+\Vert F\Vert_{L_t^{\tilde{q}'}L_x^{\tilde{r}'}(I\times\mathbb{R}^3)}
\end{align*}
whenever $s\geq 0$, $2\leq q,\tilde{q}\leq\infty$ and $2\leq r,\tilde{r}<\infty$ obey the scaling condition
\begin{displaymath}
\frac{1}{q}+\frac{3}{r}=\frac{3}{2}-s=\frac{1}{\tilde{q}'}+\frac{3}{\tilde{r}'}-2
\end{displaymath}
and the wave admissibility condition
\begin{displaymath}
\frac{1}{q}+\frac{1}{r},\,\frac{1}{\tilde{q}}+\frac{1}{\tilde{r}}\leq\frac{1}{2}
\end{displaymath}
\end{lemma}
As a consequence we also obtain the following energy estimate.
\begin{lemma}\label{lemma:wave}
Let $s\in\mathbb{R}$, $B_0\in H^s(\mathbb{R}^3)$, $B_1\in H^{s-1}(\mathbb{R}^3)$ and $F\in L^1([0,T];H^{s-1}(\mathbb{R}^3))$, $T>0$, then 
$$B\in C([0,T];H^s(\mathbb{R}^3))\cap H^{s-1}([0,T];H^{s-1}(\mathbb{R}^3))$$
defined as in previous Lemma satisfies
\begin{equation}\label{eq:enest}\begin{split}
\Vert B\Vert_{C_tH^s_x([0, T]\times\mathbb{R}^3)}&+\Vert\partial_{t} B\Vert_{C_tH^{s-1}([0, T]\times\mathbb{R}^3)}\\
&\lesssim (1+T)(\Vert B_0\Vert_{H^s(\mathbb{R}^3)}+\Vert B_1\Vert_{H^{s-1}(\mathbb{R}^3)}+\Vert F\Vert_{L_t^1H^{s-1}_x([0, T]\times\mathbb{R}^3)}).
\end{split}
\end{equation}
\end{lemma}

We conclude this Section by recalling some results concerning the Schr\"odinger propagator associated to the magnetic Laplacian $\Delta_A$. More precisely, let $A$ be a given, time dependent, divergence-free vector field, we then consider the following initial value problem
\begin{equation}\label{eq:schrmag}
\left\{\begin{array}{l} i\partial_{t} u=-\frac12\Delta_A u\\
u(s)=f,\end{array}\right.
\end{equation}
and we study the properties of its solution.
\begin{proposition}\label{prop:prop}
Let $0<T<\infty$ and let us assume that $A\in C([0, T];H^1(\R^3))$, $\d_tA\in L^1([0, T];L^3(\R^3))$. Then there exists a unique $u\in C([0, T];H^2(\R^3))\cap C^1([0, T];L^2(\R^3))$ solution to \eqref{eq:schrmag}. Moreover, it holds
\begin{equation}\label{eq:magschr_prop}
\|u\|_{L^\infty(0, T;H^2(\R^3))}\lesssim\|f\|_{H^2}\left(1+\|A\|_{L^\infty_tH^1_x}^4\right)e^{\|\d_tA\|_{L^1_tL^3_x}}.
\end{equation}
\end{proposition}
\begin{proposition}\label{eq:Duhamellemma}
Let $A\in L^\infty([0, T];H^1(\R^3))\cap W^{1,1}([0,T];L^3)$, $f\in L^1([0, T];H^{-2}(\mathbb{R}^3))$ and let $v\in C([0,T];L^2(\mathbb{R}^3))\cap W^{1,1}([0,T];H^{-2}(\mathbb{R}^3))$ 
be solution  to
$$i\partial_{t} v=-\frac12\Delta_Av+f$$
Then for every $t_0\in [0,T]$,
$$v(t)=U_A(t,t_0)-i\int_{t_0}^t U_A(t,s)f(s)ds$$
\end{proposition}
\begin{proof}
See \cite{NW1} for a proof of Propositions \ref{prop:prop} and \ref{eq:Duhamellemma}.
\end{proof}
From Proposition \ref{prop:prop} we can then define the propagator $U_A(t, s)$ associated to \eqref{eq:schrmag}, i.e. $U_A(t, s)f=u(t)$, where $u$ is the solution in Proposition \ref{prop:prop}, and $U_A$ satisfies the following properties:
\begin{itemize}
\item $U_A(t, s)H^2\subset H^2$, for any $t, x\in[0, T]$;
\item $U_A(t, t)=\mathbb I$;
\item $U_A(t_1, t_2)U_A(t_2, t_3)=U_A(t_1, t_3)$, for any $t_1, t_2, t_3\in[0, T]$.
\end{itemize}
Moreover, by \eqref{eq:magschr_prop} we have
\begin{equation*}
\mathcal K_2:=\sup_{t, s\in[0, T]}\|U_A(t, s)\|_{H^2\to H^2}\leq\left(1+\|A\|_{L^\infty_tH^1}^4\right)e^{\|\d_tA\|_{L^1_tL^3_x}}.
\end{equation*}
From the unitarity of $U_A(t, s)$ in $L^2$, $\|U_A(t, s)f\|_{L^2}=\|f\|_{L^2}$, and by interpolation, we can then infer
\begin{equation*}
\sup_{t, s\in[0, T]}\|U_A(t, s)f\|_{H^s\to H^s}<\infty,\quad\forall\;s\in[0, 2].
\end{equation*}

\section{Local well-posedness}\label{sect:LWP}
In this section we are going to prove the local well-posedness result stated in Theorem \ref{th:1} by using a fixed point argument. We split the proof in two parts: in Proposition \ref{prop:existence} we are going to show the existence and uniqueness of a local solution by means of a fixed point argument, then Proposition \ref{prop:cont_dep} will be about the continuous dependence of the solution on the initial data.
\begin{proposition}\label{prop:existence}
Let $\gamma>\frac{3}{2}$. For all $(u_0,A_0,A_1)\in X$ there exists $T_{max}>0$ and a unique maximal solution $(u,A)$ to \eqref{eq:msc} such that $u\in C([0,T_{max});H^2(\mathbb{R}^3))$,  $A\in C([0,T_{max});H^{\frac{3}{2}}(\mathbb{R}^3)\cap C^1([0,T_{max});H^{\frac{1}{2}}(\mathbb{R}^3))$, $\diver\,A=0$. Moreover  the blowup alternative holds true. 
\end{proposition}
\begin{proof}
First of all, let us define the space
\begin{equation}\begin{split}
X_T:=\{ &(u, A)\;\textrm{s.t.}\; u\in C([0,T],H^2(\mathbb{R}^3)), A \in C([0,T],H^{\frac{3}{2}}(\mathbb{R}^3))\cap C^1([0,T],H^{\frac{1}{2}}(\mathbb{R}^3)),\\
& \diver A = 0,\Vert u\Vert_{L^\infty_t H^2_x(\mathbb{R}^3)}\leq R_1,\,\Vert A\Vert_{L^\infty_t H^{\frac{3}{2}}(\mathbb{R}^3)}+\Vert\partial_{t} A\Vert_{L^\infty_t H_x^{\frac{1}{2}}(\mathbb{R}^3)}\leq R_2\}\,,
\end{split} \end{equation}
where $R_1, R_2, T>0$ will be chosen later. It is straightforward to see that $X_T$, endowed with the distance 
\begin{equation}\label{eq:distance}
d((u_1,A_1),(u_2,A_2))=\max\{\Vert u_1-u_2\Vert_{L^\infty_t L^2_x(\mathbb{R}^3)},\Vert A_1-A_2\Vert_{L^4_t L^4_x(\mathbb{R}^3)}\}\,,
\end{equation}
is a complete metric space. We also define
\begin{equation}\label{eq:XT}
\|(u, A)\|_{X_T}:=\|u\|_{L^\infty_tH^2_x([0, T]\times\R^3)}+\|A\|_{L^\infty_tH^{3/2}_x([0, T]\times\R^3)}+\|\d_tA\|_{L^\infty_tH^{1/2}_x([0, T]\times\R^3)}.
\end{equation}
Let $(u_0, A_0, A_1)\in X$, we define the map $\Phi:X_T\to X_T$, $(u, B)=\Phi(u, A)$, $(u, A)\in X_T$, where
\begin{equation}\label{eq:psi_duham}
u(t)=U_A(t,0)u_0-i\int_0^t U_A(t,s)(\phi u+\vert u\vert^{2(\gamma-1)}u)(s)ds
\end{equation}
and
\begin{align*}
B(t)=\cos (t\sqrt{-\Delta})A_0+\frac{\sin (t\sqrt{-\Delta})}{\sqrt{-\Delta}}A_1+\int_{0}^t\frac{\sin (t\sqrt{-\Delta})}{\sqrt{-\Delta}}\mathbb PJ(u,A)(s)ds
\end{align*}
Let us first show that $\Phi$ maps $X_T$ into itself. By \eqref{eq:magschr_prop} we have that for any $s\in[0, T]$,
\begin{equation*}
\|U_A(t, s)f\|_{H^2}\leq\|f\|_{H^2}\left(1+\|A\|_{L^\infty_tH^1_x}^4\right)e^{\|\d_tA\|_{L^1_tL^3_x}}
\end{equation*}
and since $\|\d_tA\|_{L^1_tL^3_x}\lesssim T\|\d_tA\|_{L^\infty_tH^{1/2}_x}$, we have
\begin{equation*}
\|U_A(t, s)f\|_{H^2}\leq C(1+R_2^4)e^{TR_2}\|f\|_{H^2}.
\end{equation*}
Let us consider the nonlinear terms in \eqref{eq:psi_duham}.
Since $\gamma>\frac{3}{2}$ the function $z\mapsto\vert z\vert^{2(\gamma-1)}z$ is $C^2(\mathbb{C};\mathbb{C})$, then by the Sobolev embedding $H^2\hookrightarrow L^\infty$ and by Lemma \ref{lemma:leibniz} we have
\begin{equation*}
\||u|^{2(\gamma-1)}u\|_{L^\infty_tH^2_x}\lesssim\|u\|_{L^\infty_tH^2_x}^{2\gamma-1}\lesssim R_1^{2\gamma-1}.
\end{equation*}
Furthermore, from \eqref{eq:hartree_lip} we have
\begin{equation*}
\|\phi u\|_{L^\infty_tH^2_x}\lesssim\|u\|_{L^\infty_tH^{3/4}_x}^2\|u\|_{L^\infty_tH^2_x}\lesssim R_1^3.
\end{equation*}
so that by putting everything together, we obtain
\begin{align*}\Vert u\Vert_{L^\infty_tH^2_x}&\leq C_1 (1+R_2^4)\exp(TR_2)\Big(\Vert u_0\Vert_{H^2}+TR_1^{2\gamma-1}+TR_1^3\Big).
\end{align*}
On the other hand, by using the Strichartz estimates for the wave equation stated in Lemma \ref{lemma:strichwave} we have
\begin{equation*}
\|B\|_{L^\infty_tH^{3/2}_x}+\|\d_tB\|_{L^\infty_tH^{1/2}_x}\lesssim(1+T)\left(\|A_0\|_{H^{3/2}}+\|A_1\|_{H^{1/2}}+\|\mathbb PJ\|_{L^1_tH^{1/2}_x}\right).
\end{equation*}
By \eqref{eq:pjestimate} we have
\begin{equation*}
\|\mathbb PJ\|_{L^\infty_tH^{1/2}_x}\lesssim R_1^2(1+R_1),
\end{equation*}
so that
\begin{equation*}
\|B\|_{L^\infty_tH^{3/2}_x}+\|\d_tB\|_{L^\infty_tH^{1/2}_x}\leq C_2(1+T)\left(\|A_0\|_{H^{3/2}}+\|A_1\|_{H^{1/2}}+TR_1^2(1+R_1)\right).
\end{equation*}
Let us now choose $R_1, R_2, T$; without loss of generality we can assume that $T<1$. Let
\begin{equation*}
\begin{aligned}
R_2 &:=4C_2\|A_0\|_{H^{3/2}}+\|A_1\|_{H^{1/2}}\\
R_1 &:=2C_1(1+R_2^4)e^{R_2}\|u_0\|_{H^2}\\
\end{aligned}
\end{equation*}
Then
\begin{equation*}
\begin{aligned}
\Vert u\Vert_{L^\infty_t H^2_x(\mathbb{R}^3)}&\leq \frac{R_1}{2}+C_1(1+R_2^4)Te^{R_2}R_1(R_1^{2(\gamma-1)}+R_1^2)\\
\Vert B\Vert_{L^\infty_tH^{\frac{3}{2}}_x(\mathbb{R}^3)}+\Vert\partial_{t} B\Vert_{L^\infty_tH^{\frac{1}{2}}_x(\mathbb{R}^3)}&\leq\frac{R_2}{2}+2C_2R_1^2(1+R_2)T
\end{aligned}
\end{equation*}
Now by choosing $T$ such that 
\begin{equation*}
\max\left\{C_1(1+R_2^4)e^{R_2}(R_1^{2(\gamma-1)}+R_1^2)T,\frac{2C_2R_1^2(1+R_2)}{R_2}T\right\}<\frac12,
\end{equation*}

we see that $\Phi$ maps $X_T$ into itself

We now prove that, possibly choosing a smaller value for $T>0$, the map $\Phi$ is indeed a contraction on $X_T$. Let us define
\begin{equation*}\begin{aligned}
(u_1,B_1) &=\Phi(u_1,A_1)\\ 
(u_2,B_2) &=\Phi(u_2,A_2)\,.
\end{aligned}\end{equation*}
By writing the difference of the equations for $u_1$ and $u_2$ we write
\begin{equation*}
(u_1-u_2)(t)=-i\int_0^tU_A(t, s)F(s)\,ds,
\end{equation*}
where $F$ is given by
\begin{equation}\begin{split}
F&=2i(A_1-A_2)\cdot\nabla u_2+(\vert A_1\vert^2-\vert A_2\vert^2)u_2+(\phi(|u_1|^2)-\phi(|u_2|^2))u_2\\&\quad+\phi(|u_1|^2)(u_1-u_2)+\left|\vert u_1\vert^{2(\gamma-1)}u_1-\vert u_2|^{2(\gamma-1)}u_2\right|=:\sum_{j=1}^5F_j.
\end{split}\end{equation}
Hence we have
\begin{equation}\label{eq:schdiff}
\|u_1-u_2\|_{L^\infty_tL^2_x}\lesssim\sum_j\|F_j\|_{L^1_tL^2_x}.
\end{equation}
We now estimate term by term; by using H\"older's inequality and Sobolev embedding we have
\begin{equation*}
\|F_1\|_{L^1_tL^2_x}\lesssim T^{3/4}\|\nabla u_2\|_{L^\infty_tH^1_x}\|A_1-A_2\|_{L^4_{t, x}},
\end{equation*}
\begin{equation*}
\|F_2\|_{L^1_tL^2_x}\lesssim T^{3/4}\|u_2\|_{L^\infty_{t, x}}\left(\|A_1\|_{L^\infty_tH^1_x}+\|A_2\|_{L^\infty_tH^1_x}\right)\|A_1-A_2\|_{L^4_{t, x}}
\end{equation*}
By using \eqref{eq:inverselaplacianestimate}, the third term is estimated by
\begin{equation*}
\|F_3\|_{L^1_tL^2_x}\lesssim T\left(\|u_1\|_{L^\infty_tH^1_x}+\|u_2\|_{L^\infty_tH^1_x}\right)\|u_2\|_{L^\infty_tH^1_x}\|u_1-u_2\|_{L^\infty_tL^2_x}.
\end{equation*}
For the term $F_4$ we use \eqref{eq:soggelemma} and Sobolev embedding to get
\begin{equation*}
\|F_4\|_{L^1_tL^2_x}\lesssim T\|u_1\|_{L^\infty_tH^2_x}^2\|u_1-u_2\|_{L^\infty_tL^2_x}.
\end{equation*}
The last term is estimated by
\begin{align*}
\Vert F_5\Vert_{L^2(\mathbb{R}^3)}&\lesssim(\Vert u_1\Vert_{L^\infty}^{2(\gamma-1)}+\Vert u_2\Vert_{L^\infty}^{2(\gamma-1)})\Vert\p\Vert_{L^2(\mathbb{R}^3)} \lesssim R_1^{2(\gamma-1)}\Vert u_1-u_2\Vert_{L^2(\mathbb{R}^3)}\,,\end{align*}
where we used the following inequality
\begin{displaymath}
\vert\vert u_1\vert^{2(\gamma-1)}u_1-\vert u_2\vert^{2(\gamma-2)}u_2\vert\lesssim (\vert u_1\vert^{2(\gamma-1)}+\vert u_2\vert^{2(\gamma-1)})\vert u_1-u_2\vert
\end{displaymath}
By putting everything together in \eqref{eq:schdiff}, and by using H\"older's inequality in time, we obtain
\begin{equation}\label{eq:contschr}
\|u_1-u_2\|_{L^\infty_tL^2_x}\lesssim(T^{3/4}+T)C(R_1, R_2)d((u_1,A_1),(u_2,A_2)).
\end{equation}
Analogously, for $B_1$, $B_2$ we write
\begin{equation}\label{eq:wavediff}
(B_1-B_2)(t)=\int_0^t\frac{\sin((t-s)\sqrt{-\Delta})}{\sqrt{-\Delta}}G(s)\,ds,
\end{equation}
where $G=\sum_{j=1}^3G_j$ is given by:
 $$G=\mathbb P\IM\{(\overline{u_1-u_2})(\nabla-iA_1)u_1-iu_1\overline{u_2}(A_1-A_2)-(u_1-u_2)(\nabla+iA_2)\overline{u_2}\}\,.$$
Here we have used the fact that $\mathbb P(\overline{u_2}\nabla(u_1-u_2))=-\mathbb P((u_1-u_2)\nabla\overline{u_2})$.
Using the Strichartz estimates in Lemma \ref{lemma:strichwave} with $q=r=\tilde q=\tilde r=4$, we get
\begin{equation}\label{eq:stricwave}
\Vert B_1-B_2\Vert_{L^4_{t,x}}\leq \Vert G\Vert_{L^{\frac{4}{3}}_{t,x}}
\end{equation}
We estimate the three terms in $G$.
The terms $G_1$ and $G_3$ are treated similarly, by Sobolev embedding and by using \eqref{eq:nablaminusa} we have
\begin{multline*}
\|G_1\|_{L^{4/3}_{t, x}}+\|G_3\|_{L^{4/3}_{t, x}}\lesssim \\T^{3/4}\left(1+\|A_1\|_{L^\infty_tH^1_x}+\|A_2\|_{L^\infty_tH^1_x}\right)
\left(\|u_1\|_{L^\infty_tH^2_x}+\|u_2\|_{L^\infty_tH^2_x}\right)\|u_1-u_2\|_{L^\infty_tL^2_x}.
\end{multline*}
By using H\"older's inequality, $G_2$ is bounded by
\begin{equation*}
\|G_2\|_{L^{4/3}_{t, x}}\lesssim T^{1/2}\|u_1\|_{L^\infty_tH^1_x}\|u_2\|_{L^\infty_tH^1_x}\|A_1-A_2\|_{L^4_{t, x}}.
\end{equation*}
Resuming, by estimating the terms in \eqref{eq:stricwave} we obtain
\begin{equation}\label{eq:contraction2}
\|A_1-A_2\|_{L^4_{t, x}}\lesssim (T^{1/2}+T^{3/4})C(R_1, R_2)d((u_1,A_1),(u_2,A_2)).
\end{equation}
By summing up \eqref{eq:contschr} and \eqref{eq:contraction2}, we finally get
\begin{equation*}
d((u_1,B_1),(u_2,B_2))\leq (T^{1/2}+T)C(R_1, R_2)d((u_1,A_1),(u_2,A_2)).
\end{equation*}
Thus, if $T>0$ is chosen sufficiently small, then $\Phi$ is a contraction. This proves that for any initial data $(u_0, A_0, A_1)\in X$, there exists a unique local solution $(u, A)$ to \eqref{eq:msc} such that $u\in C([0, T];H^2(\R^3))$, $A\in C([0, T];H^{3/2}(\R^3))\cap C^1([0, T]; H^{1/2}(\R^3))$. By a standard argument it is straightforward to show that it may be extended to a maximal solution $(u, A)$, with $u\in C([0, T_{max});H^2(\R^3))$, $A\in C([0, T_{max});H^{3/2}(\R^3))\cap C^1([0, T_{max}); H^{1/2}(\R^3))$ and that the blow-up alternative holds true, namely if $T_{max}<\infty$ then we have
\begin{equation*}
\lim_{t\to T_{max}^-}\left(\|u(t)\|_{H^2}+\|A(t)\|_{H^{3/2}}+\|\d_tA(t)\|_{H^{1/2}}\right)=\infty.
\end{equation*}
\end{proof}
Next Proposition states the continuous dependence of solution on the initial data. Its proof goes through a series of technical lemmas and it follows this strategy: first we prove the continuous dependence for more regular solutions, then by an approximation argument we prove the general result for solutions $(u, A)\in X$. This will finish the proof of Theorem \ref{th:1}. In the remaining part of the Section we state the Proposition and the Lemmas needed to prove the continuous dependence for regular solutions. Then we show how to extend it to arbitrary solutions $(u, A)\in X$. The proofs of Lemmas \ref{lemma:1}, \ref{lemma:2} and \ref{lemma:3} will be given in the Appendix.
\begin{proposition}[Continuous dependence on the initial data]\label{prop:cont_dep}
Let $0<T<T_{max}$, then the mapping $(u_0, A_0, A_1)\mapsto (u, A, \d_tA)$, where $(u, A)$ is the solution to \eqref{eq:msc} is continuous as a mapping from $X$ to $C([0, T];X)$.
\end{proposition}
For the lemmas we consider two different solutions $(u,A)$, $(u',A')$ defined from two sets of initial data $(u_0, A_0, A_1), (u_0', A_0', A_1')\in X$ such that
\begin{equation*}
\|(u_0, A_0, A_1)\|_X, \|(u_0', A_0', A_1')\|_X\leq R.
\end{equation*}
Moreover, we are also going to exploit the uniform bounds given by the total energy of system \eqref{eq:msc},
\begin{equation}\label{eq:en}
E(t)=\int\frac12|(\nabla-iA)u|^2+\frac12|\d_tA|^2+\frac12|\nabla A|^2+\frac12|\nabla\phi|^2+\frac{1}{\gamma}|u|^{2\gamma}\,dx.
\end{equation}
It is straightforward to see that it is conserved along the flow of solutions to \eqref{eq:msc}, thus if $(u, A)$, resp. $(u', A')$, is the solutions emanated from $(u_0, A_0, A_1)$, resp. $(u_0', A_0', A_1')$, then we may consider $E>0$ such that
\begin{equation*}
E(t), E'(t)\leq E,
\end{equation*}
where $E(t)$, resp. $E'(t)$, is the total energy associated to $(u, A)$, resp. $(u', A')$.
\begin{lemma}\label{lemma:1}
Let $(u, A), (u', A')$ be solutions to \eqref{eq:msc} defined as above, then we have
\begin{equation*}
\begin{aligned}
\|u-u'\|_{L^\infty_tH^2_x}\lesssim&\|\d_t(u-u')(0)\|_{L^2}+T\|\d_tu'\|_{L^\infty_tH^2_x}\left(\|A-A'\|_{L^\infty_tH^{1/2}_x}+\|u-u'\|_{L^\infty_tL^2_x}\right)\\
&+T\|(u, A)-(u', A')\|_{X_T},
\end{aligned}
\end{equation*}
where the constant depends only on $R, E$ defined as above.
\end{lemma}
\begin{lemma}\label{lemma:2}
Let $(u, A), (u', A')$ be solutions to \eqref{eq:msc} defined as above, then we have
\begin{equation*}
\|u-u'\|_{L^\infty_tL^2_x}+\|A-A'\|_{L^\infty_tH^{1/2}_x}\lesssim\|(u_0, A_0, A_1)-(u_0', A_0', A_1')\|_{L^2\times H^{1/2}\times H^{-1/2}},
\end{equation*}
where the constant depends only on $R, E$ defined as above.
\end{lemma}
\begin{lemma}\label{lemma:3}
We have
\begin{equation*}
\|\d_tu\|_{L^\infty_tH^2_x}\lesssim\|u_0\|_{H^4}+\|A_0\|_{H^{5/2}}+\|A_1\|_{H^{3/2}},
\end{equation*}
where the constant depends only on $T, R, E$.
\end{lemma}

\emph{Proof of Proposition \ref{prop:cont_dep}}.
By combining the above Lemmas, it is possible to show the continuous dependence for solutions whose initial data are $(u_0, A_0, A_1)\in H^4\times H^{5/2}\times H^{3/2}$. Indeed Lemmas \ref{lemma:1}, \ref{lemma:2} and \ref{lemma:3} imply the following estimate
\begin{equation*}
\begin{aligned}
\|&(u, A)-(u', A')\|_{X_T}\lesssim\|(u_0, A_0, A_1)-(u_0', A_0', A_1')\|_X\\&
+T\left(\|u_0'\|_{H^4}+\|A_0'\|_{H^{5/2}}+\|A_1'\|_{H^{3/2}}\right)\|(u_0, A_0, A_1)-(u_0', A_0', A_1')\|_{L^2\times H^{1/2}\times H^{-1/2}}\\
&+T\|(u, A)-(u', A')\|_{X_T}.
\end{aligned}
\end{equation*}
A straightforward bootstrap argument yields to
\begin{equation}\label{eq:contdep}
\begin{aligned}
\|&(u, A)-(u', A')\|_{X_T}\lesssim\|(u_0, A_0, A_1)-(u_0', A_0', A_1')\|_X\\&
+T\left(\|u_0'\|_{H^4}+\|A_0'\|_{H^{5/2}}+\|A_1'\|_{H^{3/2}}\right)\|(u_0, A_0, A_1)-(u_0', A_0', A_1')\|_{L^2\times H^{1/2}\times H^{-1/2}}.
\end{aligned}
\end{equation}
Let us now consider general initial data $(u_0, A_0, A_1)\in X$ and let us consider a mollified $\eta^\delta(x)=\delta^{-3}\eta(x/\delta)$, $\delta>0$, where $\eta\in C^\infty_c(\R^3)$ is a smooth, radial function with $\int\eta=1$. We define 
$u_0^\delta=\eta^\delta\ast u_0$, $A_0^{\delta^2}=\eta^{\delta^2}\ast A_0$, $A_1^{\delta^2}=\eta^{\delta^2}\ast A_1$. 
It is straightforward to check that this definition implies
\begin{equation*}
\|u_0\|_{H^4}+\|A_0^{\delta^2}\|_{H^{5/2}}+\|A_1^{\delta^2}\|_{H^{3/2}}\lesssim\delta^{-2}\left(\|u_0\|_{H^2}+\|A_0\|_{H^{3/2}}+\|A_1\|_{H^{1/2}}\right)\\
\end{equation*}
and that
\begin{equation*}
\|u_0-u^\delta\|_{L^2}+\|A_0-A_0^{\delta^2}\|_{H^{1/2}}+\|A_1-A_1^{\delta^2}\|_{H^{-1/2}}=0(\delta^2).
\end{equation*}
By using \eqref{eq:contdep} above we then infer
\begin{equation}\label{eq:dens}
\begin{aligned}
\|&(u, A)-(u^\delta, A^{\delta^2})\|_{X_T}\lesssim\|(u_0, A_0, A_1)-(u_0^\delta, A_0^{\delta^2}, A_1^{\delta^2})\|_{X}\\
&+T\left(\|u_0^\delta\|_{H^4}+\|A_0^{\delta^2}\|_{H^{5/2}}+\|A_1^{\delta^2}\|_{H^{3/2}}\right)\|(u_0, A_0, A_1)-(u_0^\delta, A_0^{\delta^2}, A_1^{\delta^2})\|_{L^2\times H^{1/2}\times H^{-1/2}}\\
\lesssim&\|(u_0, A_0, A_1)-(u_0^\delta, A_0^{\delta^2}, A_1^{\delta^2})\|_{X}+TO(\delta^2)o(\delta^{-2}).
\end{aligned}
\end{equation}
Consequently we have that $(u^\delta, A^{\delta^2})$ converges to $(u, A)$ in $X_T$ as $\delta\to0$. Let now $\{(u_{0, n}, A_{0, n}, A_{1, n}\}\subset X$ be a sequence converging to $(u_0, A_0, A_1)\in X$. We want to prove that the solutions $(u_n, A_n)$ emanated from $(u_{0, n}, A_{0, n}, A_{1, n})$ converge to $(u, A)$ in $X_T$. To do this, we regularize the initial data by considering $(u_{0, n}^\delta, A_{0, n}^{\delta^2}, A_{1, n}^{\delta^2})$. From \eqref{eq:dens} we know that $\{(u_n^\delta, A_n^{\delta^2})\}$ converges to $(u_, A_n)$ in $X_T$, as $\delta\to0$, where $(u_n, A_n)$ is the solution to \eqref{eq:msc} with initial data $(u_{0, n}, A_{0, n}, A_{1, n})$. On the other hand, 
$\{(u_{0, n}^\delta, A_{0, n}^{\delta^2}, A_{1, n}^{\delta^2})\}$ generate regular solutions, so that by \eqref{eq:contdep} we have that $\{(u_n^\delta, A_n^{\delta^2})\}$ converges to $(u^\delta, A^{\delta^2})$ in $X_T$, for $n\to\infty$. The triangular inequality then yields the convergence of $(u_n, A_n)$ to $(u, A)$ in $X_T$.
\hfill $\square$\vspace{1cm}\newline

\section{Global existence}\label{sect:g_sol}
In the previous Section we proved the local well-posedness of \eqref{eq:msc} in $H^2\times H^{3/2}$. 
However, the presence of the power-type nonlinearity in \eqref{eq:msc} prevents us to obtain a global bound for $\Vert (u(t),A(t),\partial_{t} A(t))\Vert_X$. This is different, for example, from what can be proven in [22]. Indeed, while in the case of Hartree nonlinearity it is possible to use \eqref{eq:hartree_lip} which is linear in the higher order norm, in the case of the power-type nonlinearity one has
\begin{displaymath}
\Vert\vert u\vert^{2(\gamma-1)}u\Vert_{H^2(\mathbb{R}^3)}\lesssim\Vert u\Vert_{L^\infty(\mathbb{R}^3)}^{2(\gamma-1)}\Vert u\Vert_{H^2(\mathbb{R}^3)}\,,
\end{displaymath}
which requires to bound $u$ in $H^s(\mathbb{R}^3)$, with $s>\frac{3}{2}$. Therefore it follows that the related Gronwall type inequality becomes superlinear in the higher order norm, hence it blows up in finite time.

Our strategy to investigate global in time existence will be based on the regularization of the nonlinear terms, provided by the classical Yosida approximations of the identity. We then consider the following approximating system
\begin{equation}\label{eq:appr_ms}
\left\{\begin{aligned} iu_t^\eps=&-\frac12\Delta_{\Ae} u^\eps+\phi^\eps u^\eps+N^\eps(u^\eps)\\
\Box A^\eps=&\Je \mathbb PJ^\eps\\
u^\eps(0)=&u_0,\;
A^\eps(0)=A_0,\;
\d_t A^\eps(0)=A_1,
\end{aligned}\right.
\end{equation}
where $\Je=(I-\eps\Delta)^{-1}$, $\Ae=\Je A^\eps$, $N^\eps(u^\eps)=\Je\left(|\Je u^\eps|^{2(\gamma-1)}\Je u^\eps\right)$, $J^\eps=J(u^\eps, A^\eps)$, $\phi^\eps=\phi(|u^\eps|^2)$ and we denote $\nabla_{\Ae}=\nabla-i\Ae$.  
The total energy of this approximating system is given by
\begin{equation}\label{eq:energy}
E=\int_{\mathbb{R}^3}\Big\{\vert\nabla_{\tilde{A}^\eps}u^\eps\vert^2+\frac{1}{2}\vert\nabla\phi^\eps\vert^2+\frac{1}{2}\vert\nabla A^\eps\vert^2+\frac{1}{2}\vert\partial_{t} A^\eps\vert^2+\frac{1}{\gamma}\vert\mathcal{J}^\eps u^\eps\vert^{2\gamma}\Big\}dx
\end{equation}
which is conserved along the flow of solutions. A local well-posedness result, analogous to Theorem \ref{th:1}, can be proved for the system (\ref{eq:appr_ms}) in a straightforward way.  
\begin{proposition}\label{prop:lwpapprox}
For all $(u_0, A_0, A_1)\in X$, there exists 
$T_{max}^\eps>0$ and a unique maximal solution  $(u^\eps, A^\eps)$ to \eqref{eq:appr_ms} such that $u^\eps\in C([0, T_{max}^\eps);H^2(\R^3))$, $A^\eps\in C([0, T_{max}^\eps);H^{3/2}(\R^3))\cap C^1([0, T_{max}^\eps);H^{1/2}(\R^3))$ and the usual blow-up alternative holds true. Moreover, the solution depends continuously on the initial data.
\end{proposition}
\begin{proof}
We only remark here that the local well-posedness result for system \eqref{eq:appr_ms} holds for any $\gamma\in(1, \infty)$, while in Theorem \ref{th:1} we restrict the range to $\gamma\in(\frac32, \infty)$. Indeed, because of the Yosida regularisation, we have
\begin{equation*}
\|N^\eps(u^\eps)\|_{H^2}\lesssim\||\Je u^\eps|^{2(\gamma-1)}\Je u^\eps\|_{L^2}\lesssim\|u^\eps\|_{H^2}^{2\gamma-1}.
\end{equation*}
\end{proof}
The regularisation of the nonlinear terms yields indeed the global existence of solutions.
\begin{proposition}\label{prop:gl_ex}
The solution obtained in Proposition (\ref{prop:lwpapprox}) exists globally in time, namely  $\|(u^\eps(t), A^\eps(t), \d_tA^\eps(t))\|_X$ is finite for any $t\in \R$.
\end{proposition}
The proof of Proposition \ref{prop:gl_ex} is based on the following
\begin{lemma}\label{lemma:apriori}
Let $\eps>0$, then for every $t\in\R$,
\begin{equation}\label{eq:aprioriestimate}
\|u^\eps(t)\|_{H^2}\leq C(\|u_0\|_{L^2}, E)e^{t\|\d_tA^\epsilon\|_{L^\infty_tH^{1/2}_x}}.
\end{equation}
\end{lemma}
\begin{proof}
By \eqref{eq:H2tomagLap} we have
\begin{equation*}
\begin{aligned}
\|u^\eps\|_{H^2}\lesssim&\|\Delta_{\Ae}u^\eps\|_{L^2}+\|A^\eps\|_{H^1}^4\|u^\eps\|_{L^2}\\
\leq&C(\|u_0\|_{L^2}, E)\|\Delta_{\Ae}u^\eps\|_{L^2}\,,
\end{aligned}
\end{equation*}
therefore it is convenient to estimate the norm $\|\Delta_{\Ae}u^\eps\|_{L^2}$ instead of $\Vert u\Vert_{H^2(\mathbb{R}^3)}$. By a standard energy method it follows that
\begin{equation*}
\frac{d}{dt}\left(\|(\Delta_{\Ae}u^\eps)(t)\|_{L^2}\right)\leq\|\Delta_{\Ae}(\phi^\eps u^\eps)\|_{L^2}+\|\Delta_{\Ae}N^\eps(u^\eps)\|_{L^2}+\|[\d_t, \Delta_{\Ae}]u^\eps\|_{L^2}.
\end{equation*}
The first term can be estimated by using \eqref{eq:magLaptoH2} and \eqref{eq:hartree_lip},
\begin{equation*}
\begin{aligned}
\|\Delta_{\Ae}(\phi^\eps u^\eps)\|_{L^2}\lesssim&\|\phi^\eps u^\eps\|_{H^2}+\|A^\eps\|_{H^1}^4\|\phi^\eps u^\eps\|_{L^2}\\
\lesssim&\|u^\eps\|_{H^{3/4}}^2\left(\|\Delta_{\Ae}u^\eps\|_{L^2}+\|A^\eps\|_{H^1}^4\|u^\eps\|_{L^2}\right)\\
\leq&C(\|u_0\|_{L^2}, E)\|\Delta_{\Ae}u^\eps\|_{L^2}.
\end{aligned}
\end{equation*}
The nonlinear term $N^\eps(u^\eps)$ can be controlled by exploting the regularization given by $\Je$ 
\begin{equation*}
\begin{aligned}
\|\Delta_{\Ae}N^\eps(u^\eps)\|_{L^2}\lesssim&\|N^\eps(u^\eps)\|_{H^2}+\|A^\eps\|_{H^1}^4\|N^\eps(u^\eps)\|_{L^2}\\
\lesssim&\||\Je u^\eps|^{2(\gamma-1)}\Je u^\eps\|_{L^2}+\|A^\eps\|_{H^1}^4\|N^\eps(u^\eps)\|_{L^2}\\
\lesssim&\|u^\eps\|_{H^1}^{2(\gamma-1)}\left(\|u^\eps\|_{H^2}+\|A^\eps\|_{H^1}^4\|u^\eps\|_{L^2}\right),
\end{aligned}
\end{equation*}
where the last inequality follows from Sobolev embedding. The commutator $[\d_t, \Delta_{\Ae}]u^\eps=2\d_t\Ae(\nabla+i\Ae)u^\eps$ can be estimated by using the H\"older's inequality and  the Sobolev embedding
\begin{equation*}
\begin{aligned}
\|\d_t\Ae\cdot(\nabla+i\Ae)u^\eps\|_{L^2}\leq&\|\d_t\Ae\|_{L^3}\|(\nabla+i\Ae)u^\eps\|_{L^6}\\
\lesssim&\|\d_tA^\eps\|_{H^{1/2}}\left(\|\Delta_{\Ae}u^\eps\|_{L^2}+\|A^\eps\|_{H^1}^4\|u^\eps\|_{L^2}\right).
\end{aligned}
\end{equation*}
By summing up the previous three terms 
\begin{equation*}
\frac{d}{dt}\left(\|(\Delta_{\Ae}u^\eps)(t)\|_{L^2}\right)\leq C(\|u_0\|_{L^2}, E)\|\d_tA^\eps\|_{H^{1/2}}\|\Delta_{\Ae}u^\eps\|_{L^2}\,,
\end{equation*}
hence \eqref{eq:aprioriestimate}.
\end{proof}

\emph{Proof of Proposition \ref{prop:gl_ex}}.
In order to get a bound on the $H^2$ norm of the approximating solution $u^\epsilon$, by Lemma \ref{lemma:apriori} it is sufficient to control $\|\d_tA^\eps\|_{L^\infty_tH^{1/2}_x}$. Using the energy estimate for the wave equation 
\begin{equation*}
\|A^\eps\|_{L^\infty_tH^{3/2}_x}+\|\d_tA^\eps\|_{L^\infty_tH^{1/2}_x}\lesssim C(T)\left(\|A_0\|_{H^{3/2}}+\|A_1\|_{H^{1/2}}+\|\Je\mathbb P J^\eps\|_{L^\infty_tH^{1/2}_x}\right)\,,
\end{equation*}
and, by exploiting the Yosida regularization, we get
\begin{equation*}
\|\Je\mathbb PJ^\eps\|_{L^\infty_tH^{1/2}_x}\lesssim\|\mathbb PJ^\eps\|_{L^\infty_tH^{-1/2}_x}\lesssim\|J^\eps\|_{L^\infty_tL^{3/2}_x}\leq C(E).
\end{equation*}
It follows that $\|A^\eps(t)\|_{H^{3/2}}+\|\d_tA^\eps(t)\|_{H^{1/2}}$ is uniformly bounded on compact time intervals and consequently by \eqref{eq:aprioriestimate} also $\|u^\eps(t)\|_{H^2}$ is finite. Hence, by the blow-up alternative, the solution $(u^\eps, A^\eps)$ to (\ref{eq:appr_ms}) exists globally in time.
\hfill $\square$\vspace{1cm}

Now we conclude the proof of Theorem \ref{th:3} by showing that $(u^\eps, A^\eps)$ converges to a solution to \eqref{eq:msc}, as $\eps\to0$. This will conclude the proof of Theorem (\ref{th:3}).
\\The conservation of mass and energy yields the following a priori bounds
\begin{equation}\label{eq:apr}
\begin{aligned}
&\|u^\eps\|_{L^\infty_tH^1_x(\R\times\R^3)}\leq C,\\
&\|A^\eps\|_{L^\infty_tH^1_x(\R\times\R^3)}\leq C,\quad\|\d_tA^\eps\|_{L^\infty_tL^2_x(\R\times\R^3)}\leq C\,,
\end{aligned}
\end{equation}
which imply that, up to subsequences, there exist $u\in L^\infty_tH^1_x$, $A\in L^\infty_tH^1_x\cap W^{1, \infty}_tL^2_x$, such that
\begin{align}
\label{eq:uweak}u^\eps\stackrel{*}{\rightharpoonup}u\,\,\,\,\,&\textrm{in}\,\,\,\,\,L^\infty_t H^1_x(\mathbb{R}\times\mathbb{R}^3)\\
\label{eq:aweak}A^\eps\stackrel{*}{\rightharpoonup} A\,\,\,\,\,&\textrm{in}\,\,\,\,\,L^\infty_t  H^1_x(\mathbb{R}\times\mathbb{R}^3)\\
\label{eq:ptaweak}\partial_{t} A^\eps\stackrel{*}{\rightharpoonup}\partial_{t} A\,\,\,\,\,&\textrm{in}\,\,\,\,\,L^\infty_t L^2_x(\mathbb{R}\times\mathbb{R}^3)
\end{align}
\begin{proposition}\label{prop:passage_limit}
The weak limit $(u,A)$ in (\ref{eq:uweak}), (\ref{eq:aweak}) is a finite energy weak solution to the Cauchy problem (\ref{eq:msc}), with initial datum $(u_0,A_0,A_1)$.
\end{proposition}
\begin{proof}
Let us consider $u^\eps$, by using equation \eqref{eq:appr_ms} and the a priori bounds given by the energy we have $\{\d_tu^\eps\}$ is uniformly bounded in $L^\infty(\R;H^{-1}(\R^3))$. Hence, by using the Aubin-Lions lemma and from the assumption $1<\gamma<3$ we may infer
\begin{equation}\label{eq:ustrong}
u^\eps\to u\,\,\,\,\,\textrm{in}\,\,\,\,\,L^4_{loc}(\R\times\R^3)\cap L^{2\gamma}_{loc}(\R\times\R^3)\,.
\end{equation}
This also implies that $|u^\eps|^2\rightharpoonup|u|^2$ in $L^2_tL^{6/5}_x$, and consequently from Hardy-Littlewood-Sobolev we obtain
\begin{equation}\label{eq:hartreeconv}
(-\Delta)^{-1}(|u^\eps|^2)\rightharpoonup(-\Delta)^{-1}(|u|^2),\quad\textrm{in}\,L^2_tL^6_x.
\end{equation}
Analogously for $A^\eps$, the a priori bounds yield
\begin{equation}\label{eq:astrong}
A^\eps\to A\,\,\,\,\,\textrm{in}\,\,\,\,\, L^4_{loc}(\R\times\R^3)\,.
\end{equation}
We are now able to show the convergence for the nonlinear terms $K^\eps(u^\eps, A^\eps)$, $N^\eps(u^\eps)$, $\Je\mathbb PJ^\eps$, where
\begin{equation*}
K^\eps(u^\eps,\tilde{A}^\eps)=i\tilde{A}^\eps\cdot\nabla u^\eps+\frac12\vert\tilde{A}^\eps\vert^2 u^\eps+\phi(u^\eps)u^\eps\,,
\end{equation*}
Indeed, by using the convergences \eqref{eq:uweak}-\eqref{eq:astrong} we may conclude
\begin{align*}
K^\eps(u^\eps,\tilde{A}^\eps) &\rightharpoonup K(u,A)\,\,\,\,\,\textrm{in}\,\,\,\,\, L^\frac{4}{3}_{loc}(\R\times\R^3)\,,\\
\mathbb PJ(u^\eps,\Ae)&\rightharpoonup \mathbb PJ(u,A)\,\,\,\,\,\textrm{in}\,\,\,\,\, L^\frac{4}{3}_{loc}(\R\times\R^3)\,,\\
N^\eps(u^\eps)&\rightharpoonup N(u),\quad\textrm{in}\;L^{\frac{2\gamma}{2\gamma-1}}_{loc}(\R\times\R^3).
\end{align*}

It remains to see that the initial condition is satisfied. We have that $\partial_{t} A\in L^\infty_t L^2_x(\mathbb{R}\times\mathbb{R}^3)$ and $\partial_{t}^2 A, \partial_{t} u\in L^\infty_t H^{-1}_x(\mathbb{R}\times\mathbb{R}^3)$, and consequently 
$(u, A, \d_tA)\in C(\R;H^{-1}\times L^2\times H^{-1})$. Moreover, the energy bounds imply $(u, A, \d_tA)\in L^\infty(\R;H^1\times H^1\times L^2)$ and hence we may also infer the weak continuity 
$(u, A, \d_tA)\in C_w(\R;H^1\times H^1\times L^2)$. 
\\Since $A^\eps\in L_2(0,T;H^1(\mathbb{R}^3))$ and $\partial_{t} A^\eps\in L^2(0,T;L^2(\mathbb{R}^3))$, integrating by parts we have 
$$\int_0^T\left\langle A^\eps(t)\partial_{t} f(t)+\partial_{t} A^\eps(t)f(t),\varphi\right\rangle_{H^1,H^{-1}}ds=-\left\langle A_0,\varphi\right\rangle$$
for every $\varphi\in L^2(\mathbb{R}^3)$ and all $f\in C^\infty(\mathbb{R})$ with $f(0)=1$ and $f(T)=0$. As $\eps\to 0$ we obtain
$$\int_0^T\{A(t)\partial_{t} f(t)+\partial_{t} A(t) h(t)\}\,dt=-A_0$$
in $L^2(\mathbb{R}^3)$, which implies $A|_{t=0}=A_0$.
Now we have
$$\int_0^T\left\langle\partial_{t} A^\eps\partial_{t} f(t)+\{\Delta A^\eps-\mathbb PJ(u^\eps,\Ae)\}f(t),\eta\right\rangle=\left\langle A_1,\eta\right\rangle\,,$$
and as $\eps\to 0$, we find
$$\int_0^T\{\partial_{t} A(t)\partial_{t} f(t)+\partial_{t}^2 A(t)f(t)\}=A_1$$
in $H^{-1}(\mathbb{R}^3)$, which gives us $\partial_{t} A|_{t=0}=A_1$. Applying the same argument to $u^\eps$ we deduce that $u|_{t=0}=u_0$.
\end{proof}
\section{Quantum Magnetohydrodynamics}\label{sect:QMHD}
Our last Section is devoted to point out the relation between the nonlinear Maxwell-Schr\"odinger system \eqref{eq:msc} and quantum magnetohydrodynamic (QMHD) models. Such hydrodynamic systems have been introduced in the physics literature, motivated by various applications to semiconductor devices, dense astrophysical plasmas (e.g. in white dwarfs), or laser plasmas \cite{Haas, HaasB, SE, ShuEl}. As a simplification, let us consider a one-spiecies charged quantum fluid 
with self-generated electromagnetic fields. The dynamics is described by the following system
\begin{equation}\label{eq:QMHD}
\left\{\begin{aligned}
&\d_t\rho+\diver J=0\\
&\d_tJ+\diver\left(\frac{J\otimes J}{\rho}\right)+\nabla P(\rho)=\rho E+J\wedge B+\frac12\rho\nabla\left(\frac{\Delta\sqrt{\rho}}{\sqrt{\rho}}\right),
\end{aligned}\right.
\end{equation}
where $\rho$ denotes the charge density and $J$ the current density of the quantum fluid. Here all the constants are normalized to one. The pressure term $P(\rho)$ is assumed to be isentropic of the form
 $P(\rho)=\frac{\gamma-1}{\gamma}\rho^\gamma$, $1<\gamma<3$. 
 The last term in the equation for the current density can be written in different ways
  \begin{equation}\label{eq:bohm}
 \frac12\rho\nabla\left(\frac{\Delta\sqrt{\rho}}{\sqrt{\rho}}\right)=\frac14\nabla\Delta\rho-\diver(\nabla\sqrt{\rho}\otimes\nabla\sqrt{\rho})=\frac14\diver(\rho\nabla^2\log\rho).
 \end{equation}
 and it can be seen as a self-consistent quantum potential (the so called Bohm potential) or as a quantum correction to the stress tensor. Mathematically speaking, this is a third order nonlinear dispersive term.
 The hydrodynamical system above is complemented by the Maxwell equations for the electromagnetic fields $E$ and $B$
 \begin{equation}\label{eq:Max}
 \left\{\begin{aligned}
 &\diver E=\rho,\quad\nabla\wedge E=-\d_tB\\
 &\diver B=0,\quad\nabla\wedge B=J+\d_tE.
 \end{aligned}\right.
 \end{equation}
In recent years a global existence theory of finite energy weak solutions for a class a quantum hydrodynamic systems has been established by the first and third author of this paper in \cite{AM1, AM2, AMContMath}. By means of a polar factorization techinque it is possible to define the hydrodynamic quantities by considering the Madelung transform of a wave function solution to a nonlinear Schr\"odinger equation. In this way the definition of the velocity field in the nodal regions is no longer needed. We also mention in the $H^2$ case the construction given in \cite{GaM}. Furthermore it could be interesting to consider also confining potentials as in \cite{AnCS}, generated by external magnetic fields.
The aim of this Section is to show the existence of a finite energy weak solution to \eqref{eq:QMHD}-\eqref{eq:Max} by taking advantage of our results on the system \eqref{eq:msc}.
\begin{definition}\label{def:FEWS}
Let $\rho_0, J_0, E_0, B_0\in L^1_{loc}(\R^3)$, then a finite energy weak solution to system \eqref{eq:QMHD}-\eqref{eq:Max} in the space-time slab $[0, T)\times\R^3$ is given by a quadruple $(\sqrt{\rho}, \Lambda, \phi, A)$ such that
\begin{enumerate}
\item $\sqrt{\rho}\in L^\infty([0, T);H^1(\R^3))$, $\Lambda\in L^\infty([0, T);L^2(\R^3))$, $\phi\in L^\infty([0, T);H^1(\R^3))$, $A\in L^\infty([0, T);H^1(\R^3))\cap W^{1, \infty}([0, T);L^2(\R^3))$;
\item $\rho:=(\sqrt{\rho})^2$, $J:=\sqrt{\rho}\Lambda$, $E:=-\d_tA-\nabla\phi$, $B:=\nabla\wedge A$;
\item $J\in L^2([0, T);L^2_{loc}(\R^3))$;
\item $\forall\;\eta\in C^\infty_c([0, T)\times\R^3)$,
\begin{equation*}
\int_0^T\int_{\R^3}\rho\d_t\eta+J\cdot\nabla\eta\,dxdt+\int_{\R^3}\rho_{0}(x)\eta(0, x)\,dx=0;
\end{equation*}
\item $\forall\zeta\in C^\infty_c([0, T)\times\R^3;\R^3)$,
\begin{equation*}
\begin{aligned}
\int_0^T\int_{\R^3}&J\cdot\d_t\zeta+\Lambda\otimes\Lambda:\nabla\zeta+P(\rho)\diver\zeta
+\rho E\cdot\zeta+(J\wedge B)\cdot\zeta\\&+\nabla\sqrt{\rho}\otimes\nabla\sqrt{\rho}:\nabla\zeta
+\frac{1}{4}\rho\Delta\diver\zeta\,dxdt
+\int_{\R^3}J_{0}(x)\cdot\zeta(0, X)\,dx=0;
\end{aligned}
\end{equation*}
\item $E, B$ satisfy \eqref{eq:Max} in $[0, T)\times\R^3$ in the sense of distributions;
\item (finite energy) the total mass and energy defined by
\begin{equation}\label{eq:QHD_mass} 
M(t):=\int_{\R^3}\rho(t, x)\,dx,
\end{equation}
\begin{equation}\label{eq:QMHD_en}
E(t)=\int_{\R^3}\frac12|\nabla\sqrt{\rho}|^2+\frac12|\Lambda|^2+f(\rho)
+\frac12|\d_tA|^2+\frac12|\nabla A|^2+\frac12|\nabla\phi|^2\,dx
\end{equation}
respectively, are finite for every $t\in[0, T)$. Here $f(\rho)=\frac{1}{\gamma}\rho^\gamma$.

\end{enumerate}
\end{definition}
\begin{proposition}\label{prop:QMHD}
Let $(\rho_0, J_0, B_0, E_0)$ be such that $\rho_0:=|u_0|^2$, $J_0:=\RE(\bar u_0(-i\nabla-A_0)u_0)$, $B_0:=\nabla\wedge A_0$, $E_0:=-A_1-\nabla\phi_0$, $\phi_0:=(-\Delta)^{-1}|u_0|^2$ for some 
$(u_0, A_0, A_1)\in X$, then there exists  $T_{max}>0$ such that $(\sqrt{\rho}, \Lambda, \phi, A)$ is a finite energy weak solution to \eqref{eq:QMHD}-\eqref{eq:Max} with initial data $(\rho_0, J_0, B_0, E_0)$ in the space-time slab $[0, T_{max})\times\R^3$. Moreover, the energy is conserved for all $t\in[0, T_{max})$.
\end{proposition}
To prove this Proposition we are going to use a polar factorization argument, in analogy with the electrostatic case treated in \cite{AM1, AM2}.
\newline
Given any complex valued fuction $u\in H^1(\R^3)$, we may define the set of its polar factors as 
\begin{equation*}
P(u):=\{\varphi\in L^\infty(\R^3)\;:\;\|\varphi\|_{L^\infty}\leq1,u=\sqrt{\rho}\varphi\;\;{\rm a.e. \; in}\;\R^3\},
\end{equation*}
where $\sqrt{\rho}:=|u|$. Thus, for any $\varphi\in P(u)$, we have $|\varphi|=1$ $\sqrt{\rho}\,dx$ a.e. in $\R^3$ and $\varphi$ is uniquely defined $\sqrt{\rho}\,dx$ a.e. in $\R^3$. Clearly the polar factor is not uniquely defined in the nodal regions, i.e. in the set $\{\rho=0\}$.
\newline
In the following Lemma we exploit the polar factorization of a given wave function $\psi$ in order to define the hydrodynamical quantities associated to $\psi$. This approach overcomes the WKB ansatz in the finite energy framework and allows to define the hydrodynamical quantities almost everywhere in the space, without passing through the construction of the velocity field, which is not uniquely defined in the nodal region. Furthermore, we show how this definition which uses the polar factorization is stable in $H^1(\R^3)$.
\begin{lemma}\label{lemma:polar}
Let $u\in H^1(\R^3)$, $A\in L^3(\R^3)$, and let $\sqrt{\rho}:=|u|$, $\varphi\in P(u)$. Let us define $\Lambda:=\RE(\bar\varphi(-i\nabla-A)u)\in L^2(\R^3)$, then we have
\begin{itemize}
\item $\sqrt{\rho}\in H^1(\R^3)$ and $\nabla\sqrt{\rho}=\RE(\bar\varphi\nabla u)$;
\item the following identity holds a.e. in $\R^3$,
\begin{equation}\label{eq:bil}
\RE\{\overline{(-i\nabla-A)u}\otimes(-i\nabla-A)u\}=\nabla\sqrt{\rho}\otimes\nabla\sqrt{\rho}+\Lambda\otimes\Lambda.
\end{equation}
\end{itemize}
Moreover, let $\{u_n\}\subset H^1(\R^3)$, $\{A_n\}\subset L^3(\R^3)$ be such that $u_n$ converges strongly to $u$ in $H^1$ and $A_n$ converges strongly to $A$ in $L^3$, then we have
\begin{equation*}
\nabla\sqrt{\rho_n}\to\nabla\sqrt{\rho},\quad\Lambda_n\to\Lambda,\quad{\rm in}\;L^2(\R^3),
\end{equation*}
where $\sqrt{\rho_n}:=|u_n|$, $\Lambda_n:=\RE(\bar\varphi_n(-i\nabla-A_n)u_n)$.
\end{lemma}
\begin{proof}
Let $u\in H^1(\R^3)$ and let us consider a sequence of smooth functions converging to $u$, $\{u_n\}\subset C^\infty_c(\R^3)$, $u_n\to u$ in $H^1(\R^3)$. For each $u_n$ we may define
\begin{equation*}
\varphi_n(x):=\left\{\begin{aligned}
&\frac{u_n(x)}{|u_n(x)|}&{\rm if}\;u_n(x)\neq0\\
&0&{\rm if}\;u_n(x)=0.
\end{aligned}\right.
\end{equation*}
The $\varphi_n$'s are clearly polar factors for the wave functions $u_n$. Since $\|\varphi_n\|_{L^\infty}\leq1$, then (up to subsequences) there exists $\varphi\in L^\infty(\R^3)$ such that
\begin{equation}\label{eq:pol_wstar}
\varphi_n\wstar\varphi,\quad L^\infty(\R^d).
\end{equation}
It is easy to check that $\varphi$ is indeed a polar factor for $u$. Since $\{u_n\}\subset C^\infty_c(\R^3)$, we have
\begin{equation*}
\nabla\sqrt{\rho_n}=\RE(\bar\varphi_n\nabla u_n),\quad{\rm a.e.\;in}\;\R^3.
\end{equation*}
It follows from the convergence above
\begin{equation*}
\begin{aligned}
\nabla\sqrt{\rho_n}&\to\nabla\sqrt{\rho},\quad L^2(\R^3)\\
\RE(\bar\varphi_n\nabla u_n)&\rightharpoonup\RE(\bar\varphi\nabla u),\quad L^2(\R^3),
\end{aligned}
\end{equation*}
thus $\nabla\sqrt{\rho}=\RE(\bar\varphi\nabla u)$ in $L^2(\R^3)$ and consequently the equality holds a.e. in $\R^3$.
\newline
It should be noted that here we have $\nabla\sqrt{\rho}=\RE(\bar\varphi\nabla u)$, where $\varphi$ is the weak$-\ast$ limit in \eqref{eq:pol_wstar}. However the identity above for $\nabla\sqrt{\rho}$ does not depend on the choice of $\varphi$. Indeed, by Theorem 6.19 in \cite{LL} we have $\nabla u=0$ for almost every $x\in u^{-1}(\{0\})$ and, on the other hand, $\varphi$ is uniquely determined on $\{x\in\R^3:|u(x)|>0\}$ almost everywhere. Consequently, for any 
$\varphi_1, \varphi_2\in P(u)$, we have $\RE(\bar\varphi_1\nabla u)=\RE(\bar\varphi_2\nabla u)=\nabla\sqrt{\rho}$.
The same argument applies for $\Lambda:=\RE(\bar\varphi(-\nabla-A)u)$. Let us now prove the identity \eqref{eq:bil}. Recall that we have $|\varphi|=1$ $\sqrt{\rho}\,dx$ a.e. in $\R^3$, hence again by invoking Theorem 6.19 in \cite{LL} we have
\begin{equation*}
\begin{aligned}
\RE\{\overline{(-i\nabla-A)u}\otimes(-i\nabla-A)u\}=&\RE\left\{\left(\varphi\overline{(-i\nabla-A)u}\right)\otimes\left(\bar\varphi(-i\nabla-A)u\right)\right\}\\
=&\RE\{\varphi\overline{(-i\nabla-A)u}\}\otimes\RE\{\bar\varphi(-i\nabla-A)u\}\\&-\IM\{\varphi\overline{(-i\nabla-A)u}\}\otimes\IM\{\bar\varphi(-i\nabla-A)u\}\\
=&\Lambda\otimes\Lambda+\nabla\sqrt{\rho}\otimes\nabla\sqrt{\rho},
\end{aligned}
\end{equation*}
a.e. in $\R^3$. Furthermore, by taking the trace on both sides of the above equality we furthermore obtain
\begin{equation}\label{eq:1022}
|(-i\nabla-A) u|^2=|\nabla\sqrt{\rho}|^2+|\Lambda|^2.
\end{equation}
For the second part of the Lemma, let us consider a sequence $\{u_n\}\subset H^1$ strongly converging to $u\in H^1$ and vector fields $\{A_n\}\subset L^3$ strongly converging to $A\in L^3$. As before it is straightforward to show that
\begin{equation*}
\begin{aligned}
\RE(\bar\varphi_n\nabla u_n)\rightharpoonup&\RE(\bar\varphi\nabla u),\quad L^2\\
\RE(\bar\varphi_n(-i\nabla-A_n)u_n)\rightharpoonup&\RE(\bar\varphi(-i\nabla-A)u),\quad L^2.
\end{aligned}
\end{equation*}
Moreover, from \eqref{eq:1022}, the strong convergence of $u_n$ and the weak convergence for $\nabla\sqrt{\rho_n}, \Lambda_n$, we obtain
\begin{equation*}
\begin{aligned}
\|(-i\nabla-A)u\|_{L^2}^2=&\|\nabla\sqrt{\rho}\|_{L^2}^2+\|\Lambda\|_{L^2}^2
\leq\liminf_{n\to\infty}\left(\|\nabla\sqrt{\rho_n}\|_{L^2}^2+\|\Lambda_n\|_{L^2}^2\right)\\
&=\lim_{n\to\infty}\|(-i\nabla-A_n)u_n\|_{L^2}^2=\|(-i\nabla-A)u\|_{L^2}^2.
\end{aligned}
\end{equation*}
Hence, we obtain $\|\nabla\sqrt{\rho_n}\|_{L^2}\to\|\nabla\sqrt{\rho}\|_{L^2}$ and 
$\|\Lambda_n\|_{L^2}\to\|\Lambda\|_{L^2}$. Consequently, from the weak convergence in $L^2$ and the convergence of the $L^2$ norms we may infer the strong convergence
\begin{equation*}
\nabla\sqrt{\rho_n}\to\nabla\sqrt{\rho},\quad\Lambda_n\to\Lambda,\quad{\rm in}\;L^2(\R^3).
\end{equation*}
\end{proof}
In view of Lemma \ref{lemma:polar} we can now prove Proposition \ref{prop:QMHD}. Let $(u_0, A_0, A_1)\in X$ be given, then by our main Theorem \ref{th:1} there exists a unique solution $(u, A)$ to \eqref{eq:msc} in 
$[0, T_{max})\times\R^3$ such that $u\in C([0, T_{max});H^2(\R^3))$, $A\in C([0, T_{max}); H^{3/2}(\R^3))\cap C^1([0, T_{max});H^{1/2}(\R^3))$. Let us now define $\sqrt{\rho}:=|u|$, $\Lambda:=\RE(\bar\varphi(-i\nabla+A)u)$, where $\varphi$ is a polar factor for $u$, and let $\phi:=(-\Delta)^{-1}\rho$. By differentiating $\rho$ with respect to time we have
\begin{equation*}
\begin{aligned}
\d_t\rho=&2\RE\left\{\bar u\left(-\frac{i}{2}(-i\nabla-A)^2u-i\phi u-i|u|^{2(\gamma-1)}u\right)\right\}\\
=&\IM\left\{\bar u(-i\nabla-A)^2u\right\}\\
=&\IM\left\{-i\diver\left(\bar u(-i\nabla-A)u+\overline{(-i\nabla-A)u}\cdot(-i\nabla-A)u\right)\right\}\\
=&-\diver\left(\RE(\bar u(-i\nabla-A)u)\right).
\end{aligned}
\end{equation*}
Hence by defining $J=\RE\left(\bar u(-i\nabla-A)u\right)=\sqrt{\rho}\Lambda$ we obtain the continuity equation for $\rho$
\begin{equation*}
\d_t\rho+\diver J=0.
\end{equation*}
Now let us differentiate $J$ with respect to time,
\begin{equation*}
\begin{aligned}
\d_tJ=&\RE\left\{\left(\frac{i}{2}\overline{(-i\nabla-A)^2u}+i\phi\bar u+i|u|^{2(\gamma-1)}\bar u\right)(-i\nabla-A)u\right\}\\
&+\RE\left\{\bar u(-i\nabla-A)\left(-\frac{i}{2}(-i\nabla-A)^2u-i\phi u-i|u|^{2(\gamma-1)}u\right)\right\}-\rho\d_tA\\
=&\frac12\IM\left\{\bar u(-i\nabla-A)\left((-i\nabla-A)^2u\right)-\overline{(-i\nabla-A)^2u}(-i\nabla-A)u\right\}\\
&+\RE\left\{\bar u(\phi+|u|^{2(\gamma-1)})\nabla u-\bar u\nabla\left(\phi u+|u|^{2(\gamma-1)}u\right)\right\}-\rho\d_tA.
\end{aligned}
\end{equation*}
Now the last line equals $\rho\nabla\phi-\rho\nabla\rho^{\gamma-1}-\rho\d_tA=\rho(-\d_tA-\nabla\phi)+\nabla P(\rho)$, where $P(\rho)=\frac{\gamma-1}{\gamma}\rho^\gamma$.
After some tedious but rather straightforward calculations we may see that
\begin{multline*}
\frac12\IM\left\{\bar u(-i\nabla-A)\left((-i\nabla-A)^2u\right)-\overline{(-i\nabla-A)^2u}(-i\nabla-A)u\right\}\\
=\frac14\nabla\Delta\rho-\diver\left(\RE\left\{\overline{(-i\nabla-A)u}\otimes(-i\nabla-A)u\right\}\right)+J\wedge(\nabla\wedge A).
\end{multline*}
By putting everything together we then obtain
\begin{multline*}
\d_tJ+\diver\left(\RE\{\overline{(-i\nabla-A)u}\otimes(-i\nabla-A)u\}\right)+\nabla P(\rho)=\\\rho(-\d_tA-\nabla\phi)+J\wedge(\nabla\wedge A)+\frac14\nabla\Delta\rho.
\end{multline*}
We now use the polar factorization Lemma to infer that
\begin{equation*}
\RE\{\overline{(-i\nabla-A)u}\otimes(-i\nabla-A)u\}=\nabla\sqrt{\rho}\otimes\nabla\sqrt{\rho}+\Lambda\otimes\Lambda
\end{equation*}
and consequently we get
\begin{equation*}
\d_tJ+\diver(\Lambda\otimes\Lambda)+\nabla P(\rho)=\rho E+J\wedge B+\frac14\nabla\Delta\rho-\diver(\nabla\sqrt{\rho}\otimes\nabla\sqrt{\rho}).
\end{equation*}
By recalling identity \eqref{eq:bohm} we see that this is the equation for the current density in the QMHD system \eqref{eq:QMHD}.
The above calculations are rigorous only when $(u, A)$ are sufficiently regular, however for solutions to \eqref{eq:msc} considered in Theorem \ref{th:1} they can be rigorsouly justified in the weak sense, namely in the sense of Definition \ref{def:FEWS} by regularising the initial data and by exploiting the continuous dependence showed in Proposition \ref{prop:cont_dep} and the $H^1-$stability of the polar factorization stated in Lemma \ref{lemma:polar}.
\newline
It only remains to prove that $E, B$ satisfy the Maxwell equations, but this comes in a straightforward way from the wave equation in \eqref{eq:msc} and the definitions $E=-\d_tA-\nabla\phi$, $B=\nabla\wedge A$.
\newline
Finally we remark that for solutions $(u, A)$ to \eqref{eq:msc} considered in Theorem \ref{th:1} the total energy \eqref{eq:en} is conserved. Again by using Lemma \ref{lemma:polar} we see that the energy in \eqref{eq:en} equals the one defined in \eqref{eq:QMHD_en}
 this equals the energy defined in \eqref{eq:QMHD_en}. This concludes the proof of Proposition \ref{prop:QMHD}.
\section{Appendix - Continuous dependence}
In this Appendix we are going to prove the Lemmas \ref{lemma:1}, \ref{lemma:2}, \ref{lemma:3} used to show the continuous dependence stated in Proposition \ref{prop:cont_dep}. We consider two initial data 
$(u_0, A_0, A_1), (u_0', A_0', A_1')\in X$ such that
\begin{equation*}
\|(u_0, A_0, A_1)\|_X, \|(u_0', A_0', A_1')\|_X\leq R,
\end{equation*}
and whose energies, defined as in \eqref{eq:en}, satisfy
\begin{equation*}
E(t), E'(t)\leq E.
\end{equation*}
All throughout this Appendix we are going to denote by $(u, A), (u', A')$ the solutions to \eqref{eq:msc} emanated from $(u_0, A_0, A_1), (u_0', A_0', A_1')\in X$, respectively. First of all we are going to prove Lemma \ref{lemma:1}; we will split it into two steps, see the two Lemmas \ref{lemma:11} and \ref{lemma:12} below.
\begin{lemma}\label{lemma:11}
We have
\begin{equation}\label{eq:exchange}\begin{split}
\Vert\p\Vert_{L^\infty_t H^2_x(\mathbb{R}^3)}&\lesssim_{R,E}\Vert\partial_{t}(\p)\Vert_{L^\infty_tL^2_x(\mathbb{R}^3)}+\Vert\p\Vert_{L^\infty_tL^2_x(\mathbb{R}^3)}\\&+\Vert A-A'\Vert_{L^\infty_t H_x^\frac{1}{2}(\mathbb{R}^3)}.
\end{split}\end{equation}
\end{lemma}
\begin{proof}
Let us consider the equation for the difference $u-u'$; we have
\begin{displaymath}
i\partial_{t}(u-u')=-\Delta(u-u')+2iA\cdot\nabla(u-u')+\vert A\vert^2(u-u')+F
\end{displaymath}
where
\begin{align*}
F &=2i (A-A')\cdot\nabla u'+(\vert A\vert^2-\vert A'\vert^2)u'+(\phi(|u|^2)-\phi(|u'|^2))u'\\
&+\phi(|u|^2)(u-u')+| u\vert^{2(\gamma-1)}u-| u'\vert^{2(\gamma-1)}u'\,.
\end{align*}
This implies
\begin{align*}
\Vert\Delta(u-u')\Vert_{L^2(\mathbb{R}^3)}&\leq\Vert\partial_{t}(\p)\|_{L^2(\mathbb{R}^3)}+\Vert A\cdot\nabla(\p)\Vert_{L^2(\mathbb{R}^3)}\\&+\Vert\vert A\vert^2(\p)\Vert_{L^2(\mathbb{R}^3)}+\Vert F\Vert_{L^2(\mathbb{R}^3)}
\end{align*}
From H\"older's inequality and Sobolev embedding theorem we have
\begin{align*}
\Vert A\cdot\nabla(\p)\Vert_{L^2(\mathbb{R}^3)}&\leq\Vert A\Vert_{L^6(\mathbb{R}^3)}\Vert\nabla(\p)\Vert_{L^3(\mathbb{R}^3)}\lesssim\Vert\nabla A\Vert_{L^2(\mathbb{R}^3)}\Vert\p\Vert_{H^\frac{3}{2}}\\&
\lesssim_E\Vert\p\Vert_{H^\frac{3}{2}(\mathbb{R}^3)}
\lesssim_E\Vert\p\Vert^\frac{1}{4}_{L^2(\mathbb{R}^3)}\Vert\p\Vert_{H^2(\mathbb{R}^3)}^\frac{3}{4}\\&\lesssim_E C(\eps)\Vert\p\Vert_{L^2(\mathbb{R}^3)}+\eps\Vert\p\Vert_{H^2(\mathbb{R}^3)}\,,
\end{align*}
where we do not consider the explicit dependence of the constants on $R$ and $E$. Similarly we have
\begin{displaymath}
\Vert\vert A\vert^2(\p)\Vert_{L^2(\mathbb{R}^3)}\lesssim\Vert A\Vert_{L^6}^2\Vert\p\Vert_{H^1(\mathbb{R}^3)}\lesssim_E C(\eps)\Vert\p\Vert_{L^2(\mathbb{R}^3)}+\eps\Vert\p\Vert_{H^2(\mathbb{R}^3)}\end{displaymath}
We can deal with $F$ as already done previously, getting
\begin{displaymath}
\Vert F\Vert_{L^2(\mathbb{R}^3)}\lesssim_{R,E}\Vert A-A'\Vert_{H^\frac{1}{2}(\mathbb{R}^3)}+\Vert\p\Vert_{L^2(\mathbb{R}^3)}\end{displaymath}
Finally, putting all togheter the previous inequality, we have
\begin{align*}
\Vert u-u'\Vert_{L^\infty_t H^2_x(\mathbb{R}^3)}&\leq\Vert u-u'\Vert_{L^\infty_t L^2_x(\mathbb{R}^3)}+\Vert\Delta(u-u')\Vert_{L^\infty_t L^2_x(\mathbb{R}^3)}\\
&\lesssim_{R,E} C(\eps)\Vert u-u'\Vert_{L^\infty_t L^2_x(\mathbb{R}^3)}+\Vert\partial_{t}(u-u')\Vert_{L^\infty_t L^2_x(\mathbb{R}^3)}
\\&+\Vert A-A'\Vert_{L^\infty_t H^{\frac{1}{2}}_x(\mathbb{R}^3)}+\eps\Vert u-u'\Vert_{L^\infty_t H^2_x(\mathbb{R}^3)}\,.
\end{align*}
Now, by choosing $\eps$ sufficently small, we get  (\ref{eq:exchange}).
\end{proof}
We note that in the same way we can prove that
\begin{equation}\label{eq:reverse}
\Vert\partial_{t}(\p)\Vert_{L^\infty_tL^2_x(\mathbb{R}^3)}\lesssim_{R,E} \Vert\p\Vert_{L^\infty_tH^2_x(\mathbb{R}^3)}+\Vert A-A'\Vert_{L^\infty_t H_x^\frac{1}{2}(\mathbb{R}^3)}
\end{equation}

In order to estimate the term $\Vert\partial_{t}(\p)\Vert_{L^\infty_tL^2_x(\mathbb{R}^3)}$ we use next lemma.
\begin{lemma}\label{lemma:12}
The following inequality holds:
\begin{equation}\label{eq:stimaderivatatemporale}\begin{split}
\Vert\partial_{t} & u-\partial_{t} u'\Vert_{L^\infty_tL^2_x(\mathbb{R}^3)}\lesssim_{R,E}\Vert\partial_{t}(u-u')(0)\Vert_{L^2(\mathbb{R}^3)}\\
&+T\Vert\partial_{t} u'\Vert_{L^\infty_tH^2_x(\mathbb{R}^3)}\big(\Vert A-A'\Vert_{L^\infty_t H_x^\frac{1}{2}(\mathbb{R}^3)}+\Vert u-u'\Vert_{L^\infty_t L^2_x(\mathbb{R}^3)}\big)\\
&+T\Vert (u-u',A-A',\partial_{t} A-\partial_{t} A')\Vert_{X}
\end{split}
\end{equation}
\end{lemma}
\begin{proof}
We start by differentiating in time the equation
\begin{displaymath}
i\partial_{t} u=-\Delta_Au+\phi(u)u+| u\vert^{2(\gamma-1)}u.
\end{displaymath}
We then get
\begin{equation*}
i\partial_{t}^2u=-\Delta_A\partial_{t} u+\phi(u)\partial_{t} u+(2i\partial_{t} A(\nabla-iA)+\partial_{t}\phi)u+\partial_{t}(|u\vert^{2(\gamma-1)}u)
\end{equation*}
Writing the corresponding equation for $\partial_{t}^2u'$ and taking the difference with the previous one we get
\begin{equation}\label{eq:secondderivative}
i\partial^2_t (u-u')=-\Delta_A(\partial_{t} u-\partial_{t} u')+F\,,
\end{equation}
where $F$ is given by
\begin{equation}\label{eq:inhom}
\begin{split}
F &=\bigg[2i(A-A')\bigg(\nabla-\frac{i}{2}(A+A')\bigg)+(\phi-\phi')\bigg]\partial_{t} u'+\phi(\partial_{t} u-\partial_{t} u')\\&+(2i\partial_{t} A(\nabla-iA)+\partial_{t}\phi)(u-u')+\partial_{t}(|u\vert^{2(\gamma-1)}u-| u'\vert^{2(\gamma-1)}u')\\&+
(2i\partial_{t}(A-A')(\nabla-iA)-2i(A-A')\partial_{t} A'+\partial_{t}(\phi-\phi'))u'\,.
\end{split} \end{equation}
Using the unitarity in $L^2(\mathbb{R}^3)$ of $U_A(t,s)$ we get
\begin{equation}\label{eq:Duhamel}
\Vert\partial_{t}(u-u')(t)\Vert_{L^2(\mathbb{R}^3)}\leq\Vert\partial_{t}(u-u')(0)\Vert_{L^2}+\int_0^t\Vert F(s)\Vert_{L^2(\mathbb{R}^3)}ds\,.
\end{equation}
We estimate the inhomogenous term $F$, we have
\begin{align*}
\bigg\Vert \Big[2i(A-A')&(\nabla-\frac{i}{2}(A+A'))+(\phi-\phi')\Big]\partial_{t} u' \bigg\Vert_{L^2(\mathbb{R}^3)}\\&\lesssim_{R,E} \bigg(\Vert u-u'\Vert_{L^2(\mathbb{R}^3)}+\Vert A-A'\Vert_{H^{\frac{1}{2}}(\mathbb{R}^3)}\bigg)\Vert\partial_{t} u'\Vert_{H^2(\mathbb{R}^3)}
\end{align*}
This inequality follows from
\begin{equation*}\begin{split}
\bigg\Vert(&A-A')\bigg(\nabla-\frac{i}{2}(A+A')\bigg)\partial_{t} u'\bigg\Vert_{L^2(\mathbb{R}^3)}\\
&  \leq\Vert A-A'\Vert_{L^3(\mathbb{R}^3)}\bigg\Vert\bigg(\nabla-\frac{i}{2}(A+A')\bigg)\partial_{t} u'\bigg\Vert_{L^6(\mathbb{R}^3)}\\
&\lesssim \Vert A-A'\Vert_{H^\frac{1}{2}(\mathbb{R}^3)}\big\{\Vert\nabla\partial_{t} u'\Vert_{H^1(\mathbb{R}^3)}+\Vert A+A'\Vert_{L^6(\mathbb{R}^3)}\Vert\partial_{t} u\Vert_{L^\infty(\mathbb{R}^3)}\big\}\\
& \lesssim \Vert A-A'\Vert_{H^{\frac{1}{2}}(\mathbb{R}^3)}\Vert\partial_{t} u\Vert_{H^2(\mathbb{R}^3)}\big(1+\Vert\nabla A\Vert_{L^2(\mathbb{R}^3)}+\Vert\nabla A'\Vert_{L^2(\mathbb{R}^3)}\big)\\ 
&   \lesssim_E\Vert A-A'\Vert_{H^{\frac{1}{2}}(\mathbb{R}^3)}\Vert\partial_{t} u'\Vert_{H^2(\mathbb{R}^3)}
\end{split}\end{equation*}
and
\begin{align*}
&\Vert (\phi-\phi')\partial_{t} u'\Vert_{L^2(\mathbb{R}^3)}\leq\Vert\Delta^{-1}((u-u')\overline{u}+\overline{(u-u')}u')\partial_{t} u'\Vert_{L^2(\mathbb{R}^3)}\\&
\lesssim\Vert u-u'\Vert_{L^2(\mathbb{R}^3)}\Vert\overline{u}\Vert_{L^3(\mathbb{R}^3)}\Vert\partial_{t} u\Vert_{L^3(\mathbb{R}^3)}+\Vert\overline{(u-u')}\Vert_{L^2(\mathbb{R}^3)}\Vert u'\Vert_{L^3(\mathbb{R}^3)}\Vert\partial_{t} u\Vert_{L^3(\mathbb{R}^3)}\\&
\lesssim_R\Vert u-u'\Vert_{L^2(\mathbb{R}^3)}\Vert\partial_{t} u'\Vert_{H^2(\mathbb{R}^3)}
\end{align*}
where we used  H\"older inequality, the Sobolev embeddings $H^1(\mathbb{R}^3)\hookrightarrow L^6(\mathbb{R}^3)$, $H^\frac{1}{2}(\mathbb{R}^3)\hookrightarrow L^3(\mathbb{R}^3)$ and (\ref{eq:inverselaplacianestimate}).\\
Furthermore, from (\ref{eq:soggelemma}) we may infer
\begin{displaymath}
\Vert \phi(\partial_{t} u-\partial_{t} u')\Vert_{L^2(\mathbb{R}^3)}\lesssim_R\Vert\partial_{t} u-\partial_{t} u'\Vert_{L^2(\mathbb{R}^3)}\,.
\end{displaymath}

Again, 
\begin{align*}
\Vert 2i\partial_{t} A(\nabla-iA)(u-u')\Vert_{L^2(\mathbb{R}^3)}&\lesssim\Vert\partial_{t} A\Vert_{L^3(\mathbb{R}^3)}\Vert(\nabla-iA)(u-u')\Vert_{L^6}\\&\lesssim_{R,E}\Vert u-u'\Vert_{H^2(\mathbb{R}^3)}
\end{align*}
and, by using (\ref{eq:nablaminusa}) and (\ref{eq:inverselaplacianestimate}),
\begin{align*}
\Vert\partial_{t}\phi(u-u')\Vert_{L^2(\mathbb{R}^3)}&\lesssim\Vert(\Delta^{-1}(2\RE(\overline{u}\partial_{t} u)))(u-u')\Vert_{L^2(\mathbb{R}^3)}
\\ &\lesssim\Vert\partial_{t} u\Vert_{L^2(\mathbb{R}^3)}\Vert\overline{u}\Vert_{L^3(\mathbb{R}^3)}\Vert u-u'\Vert_{H^2(\mathbb{R}^3)}\\&\lesssim_R \Vert u-u'\Vert_{H^2(\mathbb{R}^3)}\,.
\end{align*}

Observe  that one has
\begin{displaymath}
\partial_{t}(| u\vert^{2(\gamma-1)}u)=\gamma|u\vert^{\g}\partial_{t} u+(\gamma-1)| u\vert^{2(\gamma-2)}u^2\partial_{t}\overline{u},
\end{displaymath}
therefore it follows
\begin{align*}
\partial_{t}(| u\vert^{2(\gamma-1)}u-| u'\vert^{2(\gamma-1)}u') &=
\gamma\partial_{t} u(|u\vert^{2(\gamma-1)}-|u'\vert^{2(\gamma-1)})+\gamma|u'\vert^{2(\gamma-1)}\partial_{t}(\p)\\
&+(\gamma-1)\partial_{t}\overline{u}(| u\vert^{2(\gamma-2)}u^2-|u'\vert^{2(\gamma-2)}u'^2)\\&+(\gamma-1)|u'\vert^{2(\gamma-2)}u^2\partial_{t}(\overline{\p})
\end{align*}
We then have
\begin{align*}
\|\partial_{t}(| u\vert^2u-|u'\vert^2u')\|_{L^2(\mathbb{R}^3)} \lesssim_R 
\|\partial_{t}(u-u')\|_{L^2(\mathbb{R}^3)} + 
\|u-u'\|_{H^2(\mathbb{R}^3)}\,,
\end{align*}
where we used the following two inequalities
\begin{align*}
\Big|\vert z\vert^{\g}-\vert z'\vert^{\g}\Big|&\lesssim\Big(\vert z\vert^{2\gamma-3}+\vert z'\vert^{2\gamma-3}\Big)\vert z-z'\vert \\
\Big|\vert z\vert^{2(\gamma-2)}z^2-\vert z'\vert^{2(\gamma-2)}z'^2\Big|&\lesssim \Big(\vert z\vert^{2\gamma-3}+\vert z'\vert^{2\gamma-3}\Big)\vert z-z'\vert\,.
\end{align*}
For the last term, with similar computations, we have 

\begin{displaymath}
\Vert\partial_{t}(A-A')(\nabla-iA)u'\Vert_{L^2(\mathbb{R}^3)}\lesssim_{R,E}\Vert\partial_{t}(A-A')\Vert_{H^\frac{1}{2}(\mathbb{R}^3)}
\end{displaymath}
\begin{displaymath}
\Vert\partial_{t} A'(A-A')u'\Vert_{L^2(\mathbb{R}^3)}\lesssim\Vert\partial_{t} A'\Vert_{L^3(\mathbb{R}^3)}\Vert A-A'\Vert_{L^6(\mathbb{R}^3)}\Vert u'\Vert_{L^\infty(\mathbb{R}^3)}
\end{displaymath}
\begin{displaymath}
\lesssim_R\Vert A-A'\Vert_{H^\frac{3}{2}(\mathbb{R}^3)}
\end{displaymath}
\begin{displaymath}
\Vert(\partial_{t}\phi-\partial_{t}\phi')u'\Vert_{L^2(\mathbb{R}^3)}\lesssim_R\Vert\partial_{t} u-\partial_{t} u'\Vert_{L^2(\mathbb{R}^3)}+\Vert u-u'\Vert_{H^2(\mathbb{R}^3)}\,.
\end{displaymath}
By putting everything together, we obtain
\begin{align*}
\Vert\partial_{t} & u -\partial_{t} u'\Vert_{L^\infty_tL^2_x(\mathbb{R}^3)}\lesssim_{R,E}\Vert\partial_{t}(u-u')(0)\Vert_{L^2(\mathbb{R}^3)}+T\Vert\partial_{t} u-\partial_{t} u'\Vert_{L^\infty_tL^2_x(\mathbb{R}^3)}\\
&+T\Vert\partial_{t} u'\Vert_{L^\infty_tH^2_x(\mathbb{R}^3)}\big(\Vert A-A'\Vert_{L^\infty_tH_x^\frac{1}{2}(\mathbb{R}^3)}+\Vert u-u'\Vert_{L^\infty_tL^2_x(\mathbb{R}^3)}\big)
\\&+T\big(\Vert u-u'\Vert_{L^\infty_tH^2_x(\mathbb{R}^3)}+\Vert A-A'\Vert_{L^\infty_tH^\frac{3}{2}_x(\mathbb{R}^3)}+\Vert\partial_{t}(A-A')\Vert_{L^\infty_tH^\frac{1}{2}_x(\mathbb{R}^3)}\big)\,,
\end{align*}
which gives (\ref{eq:stimaderivatatemporale}), by using (\ref{eq:reverse}) for the term $\Vert\partial_{t}(u-u')\Vert_{L^\infty_t L^2_x(\mathbb{R}^3)}$ in the righthand side of the previous inequality.
\end{proof}
By putting together the two previous Lemmas we then have Lemma \ref{lemma:1}. Now we are going to estimate the term $\Vert A-A'\Vert_{L^\infty_t H^\frac{1}{2}_x(\mathbb{R}^3)}+\Vert u-u'\Vert_{L^\infty_t L^2_x(\mathbb{R}^3)}$. 
\begin{lemma}
Let $(u, A)$, $(u', A)$ be as in previous lemmas, then
\begin{equation}\label{eq:factorderivative}
\begin{split}
\Vert A-A'\Vert_{L_t^\infty H^\frac{1}{2}(\mathbb{R}^3)}&+\Vert u-u'\Vert_{L^\infty_tL^2(\mathbb{R}^3)}\\&\lesssim_{R,E}\Vert(u_0-u_0',A_0-A_0',A_1-A_1')\Vert_{L^2(\mathbb{R}^3)\times H^{\frac{1}{2}}(\mathbb{R}^3)\times H^{-\frac{1}{2}}(\mathbb{R}^3)}
\end{split}
\end{equation}
\end{lemma}
\begin{proof}
Writing the difference equation for $A$ and $A'$ we get
\begin{align*}
\Box (A-A')=G\,,
\end{align*}
with $$G=\mathbb P\IM\{(\overline{u-u'})(\nabla-iA)u-iu\overline{u'}(A-A')-(u-u')(\nabla+iA')\overline{u'}\}$$ 
where we used the fact that $\mathbb P(\overline{u'}\nabla(u-u'))=-\mathbb P((u-u')\nabla\overline{u'})$.
By applying the energy estimate (\ref{eq:enest}) we get
\begin{align*}
\Vert A-A'\Vert_{L_t^\infty H^\frac{1}{2}_x(\mathbb{R}^3)}&\lesssim (1+T)\Vert(A_0-A_0',A_1-A_1')\Vert_{H^\frac{1}{2}(\mathbb{R}^3)\times H^{-\frac{1}{2}}(\mathbb{R}^3)}\\&+(1+T)\Vert G\Vert_{L^1_tH^{-\frac{1}{2}}_x(\mathbb{R}^3)}
\end{align*}

Using the embedding $L^{\frac{3}{2}}(\mathbb{R}^3)\hookrightarrow H^{-\frac{1}{2}}(\mathbb{R}^3)$ we have
\begin{align*}
\Vert(u-u')&(\nabla-iA)u\Vert_{L^\frac{3}{2}(\mathbb{R}^3)}\leq \Vert u-u'\Vert_{L^2(\mathbb{R}^3)}\Vert(\nabla-iA)u\Vert_{L^6(\mathbb{R}^3)}\\
&\lesssim \Vert u-u'\Vert_{L^2(\mathbb{R}^3)}\big\{\Vert\nabla u\Vert_{H^1(\mathbb{R}^3)}+\Vert Au\Vert_{L^6(\mathbb{R}^3)}\big\}\\&\lesssim\Vert u-u'\Vert_{L^2(\mathbb{R}^3)}\Vert u\Vert_{H^2(\mathbb{R}^3)}\big(1+\Vert \nabla A\Vert_{L^2(\mathbb{R}^3)}\big)\\
&\lesssim_{R,E}\Vert u-u'\Vert_{L^2(\mathbb{R}^3)}\,.
\end{align*}
Analogously
\begin{displaymath}
\Vert (u-u')(\nabla+iA)\overline{u'}\Vert_{L^{\frac{3}{2}}(\mathbb{R}^3)}\lesssim_{R,E} \Vert u-u'\Vert_{L^2(\mathbb{R}^3)}
\end{displaymath}
and
\begin{displaymath}
\Vert u\overline{u'}(A-A')\Vert_{L^\frac{3}{2}(\mathbb{R}^3)}\lesssim_R\Vert A-A'\Vert_{H^\frac{1}{2}(\mathbb{R}^3)}
\end{displaymath}
In a similar way, using the difference of the equations for $u$ and $u'$ we get
\begin{align*}
\Vert u -u'\Vert&_{L^\infty_tL^2(\mathbb{R}^3)}\lesssim_{R,E}\Vert u_0-u_0'\Vert_{L^2(\mathbb{R}^3)}\\
&+T\big\{\Vert A-A'\Vert_{L^\infty_tH^\frac{1}{2}(\mathbb{R}^3)}+\Vert u-u'\Vert_{L^\infty_tL^2(\mathbb{R}^3)}\big\}
\end{align*}
Putting all togheter, taking $T$ sufficiently small, we get (\ref{eq:factorderivative}).
\end{proof}
Now, using (\ref{eq:factorderivative}) in (\ref{eq:stimaderivatatemporale}), we get
\begin{equation}\label{eq:schrpart}
\begin{split}
&\Vert\partial_{t} u -\partial_{t} u'\Vert_{L^\infty_tL^2(\mathbb{R}^3)}\lesssim\Vert\partial_{t}(u-u')(0)\Vert_{L^2(\mathbb{R}^3)}\\
&+T\Vert\partial_{t} u'\Vert_{L^\infty_tH^2(\mathbb{R}^3)}\Vert(u_0-u_0',A_0-A_0',A_1-A_1')\Vert_{X^{0,\frac{1}{2}}}\\
&+T\Vert(u-u',A-A',\partial_{t} A-\partial_{t} A')\Vert_{X}
\end{split}
\end{equation}
On the other hand, by analogous arguments, we have the following estimate for the Maxwell part
\begin{equation}\label{eq:maxwellpart}
\begin{split}
\Vert A-A'\Vert_{L_t^\infty H^\frac{3}{2}(\mathbb{R}^3)}&+\Vert \partial_{t} A-\partial_{t} A'\Vert_{L_t^\infty H^\frac{1}{2}(\mathbb{R}^3)}\\
&\lesssim \Vert(A_0-A_0',A_1-A_1')\Vert_{H^\frac{3}{2}(\mathbb{R}^3)\times H^\frac{1}{2}(\mathbb{R}^3)}\\&
+T\Vert (u-u',A-A',\partial_{t} A-\partial_{t} A')\Vert_{X}
\end{split}
\end{equation}
In order to get the estimate for $\Vert(\p,A-A',\partial_{t} A-\partial_{t} A')\Vert_X$ we put togheter (\ref{eq:exchange}), choosing a sufficiently small $T$, (\ref{eq:schrpart}) and (\ref{eq:maxwellpart}) to get
\begin{equation}\label{eq:finalics}\begin{split}
\Vert (u-u', A &-A',\partial_{t} A-\partial_{t} A')\Vert_{X}\lesssim\Vert(u_0-u'_0,A_0-A_0', A_1-A_1')\Vert_{X}\\
&+(\Vert\partial_{t} u'\Vert_{L^\infty_tH^2(\mathbb{R}^3)}+1)\Vert(u_0-u_0',A_0-A_0',A_1-A_1')\Vert_{X^{0,\frac{1}{2}}}
\end{split}
\end{equation}
where we applied (\ref{eq:reverse}) to the term $\Vert\partial_{t}(\p)(0)\Vert_{L^2(\mathbb{R}^3)}$.\\

Finally we are going to estimate the term $\Vert\partial_{t} u'\Vert_{L^\infty_t H^2(\mathbb{R}^3)}$.
\begin{lemma} The following estimate holds:
\begin{equation}\label{eq:temporalderivativeestimate}
\Vert\partial_{t} u\Vert_{L^\infty_t H^2(\mathbb{R}^3)}\leq\Vert\partial_{t}^2u\Vert_{L^2(\mathbb{R}^3)}+C(E,R)
\end{equation}
\end{lemma}
\begin{proof}
From the equation
\begin{align*}
i\partial_{tt}^2u&=-\Delta\partial_{t} u+2iA\cdot\nabla\partial_{t} u+\vert A\vert^2\partial_{t} u+2i\partial_{t} A\cdot\nabla u+2A\cdot\partial_{t} Au
\\&+\partial_{t}\phi u+\phi\partial_{t} u+\partial_{t}(| u\vert^{2(\gamma-1)}u)
\end{align*}
we can estimate $\Vert\partial_{t} u\Vert_{H^2(\mathbb{R}^3)}$. Indeed
$$\Vert\partial_{t} u\Vert_{H^2(\mathbb{R}^3)}\leq\Vert\partial_{t} u\Vert_{L^2(\mathbb{R}^3)}+\Vert\Delta\partial_{t} u\Vert_{L^2(\mathbb{R}^3)}\leq C(R)+\Vert\Delta\partial_{t} u\Vert_{L^2(\mathbb{R}^3)}$$
So we have
\begin{align*}
\Vert\Delta\partial_{t} u\Vert_{L^2(\mathbb{R}^3)}&\leq\Vert\partial_{tt}^2u\Vert_{L^2(\mathbb{R}^3)}+\Vert A\cdot\nabla\partial_{t} u\Vert_{L^2(\mathbb{R}^3)}+\Vert \vert A\vert^2\partial_{t} u\Vert_{L^2(\mathbb{R}^3)}\\
&+\Vert\partial_{t} A\cdot\nabla u\Vert_{L^2(\mathbb{R}^3)}+\Vert A\cdot\partial_{t} Au\Vert_{L^2(\mathbb{R}^3)}\\
&+\Vert \partial_{t}\phi u+\phi\partial_{t} u\Vert_{L^2(\mathbb{R}^3)}+\Vert\partial_{t}(| u\vert^{2(\gamma-1)}u)\Vert_{L^2(\mathbb{R}^3)}
\end{align*}
We begin with the estimate of the right-hand side of the previous inequality.
\begin{align*}
&\Vert A\cdot\nabla\partial_{t} u\Vert_{L^2(\mathbb{R}^3)}\lesssim\Vert A\Vert_{L^6(\mathbb{R}^3)}\Vert\nabla\partial_{t} u\Vert_{L^3(\mathbb{R}^3)}\lesssim\Vert \nabla A\Vert_{L^2(\mathbb{R}^3)}\Vert\partial_{t} u\Vert_{H^\frac{3}{2}(\mathbb{R}^3)}\\
&\lesssim\sqrt{E}\Vert\partial_{t} u\Vert^\frac{1}{4}_{L^2(\mathbb{R}^3)}\Vert\partial_{t} u\Vert_{H^2(\mathbb{R}^3)}^\frac{3}{4}\lesssim\sqrt{E}(C(\eps)\Vert\partial_{t} u\Vert_{L^2(\mathbb{R}^3)}+\eps\Vert\partial_{t} u\Vert_{H^2(\mathbb{R}^3)})\\&
\lesssim C(E,R)+C(R)\eps\Vert\partial_{t} u\Vert_{H^2(\mathbb{R}^3)}
\end{align*}
In the same way
\begin{align*}
\Vert\vert A\vert^2\partial_{t} u\Vert_{L^2(\mathbb{R}^3)}&\lesssim\Vert A\Vert^2_{L^6(\mathbb{R}^3)}\Vert\partial_{t} u\Vert_{L^6(\mathbb{R}^3)}\lesssim\Vert\nabla A\Vert_{L^2(\mathbb{R}^3)}^2\Vert\partial_{t} u\Vert_{H^1(\mathbb{R}^3)}\\
&\lesssim(1+\eps\Vert\partial_{t} u\Vert_{H^2(\mathbb{R}^3)})
\end{align*}
The other terms are all bounded by $C(R)$; for instance
\begin{displaymath}
\Vert\partial_{t} A\nabla u\Vert_{H^2(\mathbb{R}^3)}\lesssim\Vert\partial_{t} A\Vert_{H^\frac{1}{2}(\mathbb{R}^3)}\Vert u\Vert_{H^2(\mathbb{R}^3)}\leq C(R)
\end{displaymath}
or
\begin{align*}
\Vert\partial_{t}(| u\vert^{2(\gamma-1)}u)\Vert_{L^2(\mathbb{R}^3)}&\lesssim\Vert|u\vert^{2(\gamma-1)}\partial_{t} u\Vert_{L^2(\mathbb{R}^3)}\\
&\lesssim\Vert u\Vert_{L^\infty}^{2(\gamma-1)}\Vert\partial_{t} u\Vert_{L^2(\mathbb{R}^3)}\leq C(R)
\end{align*}
We can deal with the remaining terms analogously. Finally we get
\begin{displaymath}
\Vert\partial_{t} u\Vert_{H^2(\mathbb{R}^3)}\lesssim \Vert\partial_{tt}^2u\Vert_{L^2(\mathbb{R}^3)}+C(E,R)+C(R)\eps\Vert\partial_{t} u\Vert_{H^2(\mathbb{R}^3)}
\end{displaymath}
which gives (\ref{eq:temporalderivativeestimate}) for sufficiently small $\eps$.
\end{proof}

To complete the estimates we have to deal with $\Vert\partial_{t}^2u\Vert_{L^\infty_tL^2(\mathbb{R}^3)}$.
We write the equation for the time derivative $\partial_{t}^2u$
\begin{displaymath}
i\partial_{t}^3u=-\Delta u+2iA\cdot\nabla\partial_{tt}u+\vert A\vert^2\partial_{tt}u+G
\end{displaymath}
where
\begin{align*}
G &=4i\partial_t A\cdot\nabla\partial_{t} u+4A\cdot\partial_{t} A\partial_{t} u+2i\partial_{t}^2A(\nabla u-iAu)+2(\partial_{t} A)^2u\\
&+2\partial_{t}\phi\partial_{t} u+\partial_{t}^2\phi u+\partial_{t}^2u\phi+\partial_{t}^2(| u\vert^{2(\gamma-1)}u)
\end{align*}
Using Duhamel's representation in Lemma (\ref{eq:Duhamellemma}) we have 
\begin{displaymath}
\Vert\partial_{tt}u\Vert_{L^\infty_tL^2(\mathbb{R}^3)}\lesssim\Vert\partial_{tt}u(0)\Vert_{L^2(\mathbb{R}^3)}+T\Vert G\Vert_{L^\infty_tL^2(\mathbb{R}^3)}
\end{displaymath}
Proceeding as before we finally get
\begin{align*}
\Vert\partial_{t}^2u\Vert_{L^2(\mathbb{R}^3)}&\lesssim\Vert\partial_{t}^2u(0)\Vert_{L^2(\mathbb{R}^3)}\\&+TC(R,E)\Big\{\Vert\partial_{t} u\Vert_{L^\infty_tH^2(\mathbb{R}^3)}+\Vert A\Vert_{L^\infty_tH^\frac{5}{2}(\mathbb{R}^3)}+\Vert\partial_{t} A\Vert_{L^\infty_tH^\frac{3}{2}(\mathbb{R}^3)} \Big\}
\end{align*}
We estimate the right-hand side of the previous inequality. We have
\begin{align*}
\Vert A\Vert_{L^\infty_tH^\frac{5}{2}(\mathbb{R}^3)}+\Vert\partial_{t} A\Vert_{L^\infty_tH^\frac{3}{2}(\mathbb{R}^3)}&\lesssim(1+T)\Vert(A_0,A_1)\Vert_{H^\frac{5}{2}(\mathbb{R}^3)\times H^\frac{3}{2}(\mathbb{R}^3)}\\
&+T(1+T)\Vert J\Vert_{L^\infty_tH^\frac{3}{2}(\mathbb{R}^3)}
\end{align*}
For the term with $J$, proceeding as in (\ref{eq:pjestimate}),  we have
\begin{displaymath}
\Vert\overline{u}\nabla u\Vert_{H^\frac{3}{2}(\mathbb{R}^3)}\lesssim C(R)\Vert u\Vert_{H^4(\mathbb{R}^3)}
\end{displaymath}
and
\begin{displaymath}
\Vert A\Vert u\vert^2\Vert_{H^\frac{3}{2}(\mathbb{R}^3)}\lesssim C(R,E)
\end{displaymath}
So
\begin{align*}
\Vert A\Vert_{L^\infty_tH^\frac{5}{2}(\mathbb{R}^3)}+\Vert\partial_{t} A\Vert_{L^\infty_tH^\frac{3}{2}(\mathbb{R}^3)}&\lesssim(1+T)\Vert(A_0,A_1)\Vert_{H^\frac{5}{2}(\mathbb{R}^3)\times H^\frac{3}{2}(\mathbb{R}^3)}\\&+T(1+T)\Vert u\Vert_{L^\infty H^4(\mathbb{R}^3)}
\end{align*}
Moreover   since $\Vert u\Vert_{H^4(\mathbb{R}^3)}\lesssim\Vert u\Vert_{L^2(\mathbb{R}^3)}+\Vert\Delta u\Vert_{H^2(\mathbb{R}^3)}$ and by the equation for $u$ it follows
\begin{displaymath}
\Vert\Delta u\Vert_{H^2(\mathbb{R}^3)}\lesssim\Vert\partial_{t} u\Vert_{H^2(\mathbb{R}^3)}+\Vert A\cdot\nabla u+\vert A\vert^2u\Vert_{H^2}+\Vert\phi u\Vert_{H^2}+\Vert| u\vert^{2(\gamma-1)}u\Vert_{H^2(\mathbb{R}^3),}
\end{displaymath}
then by estimating the right-hand side as before, we obtain
\begin{displaymath}
\Vert u\Vert_{L^\infty_tH^4(\mathbb{R}^3)}\lesssim C(E,R) \Big(\Vert\partial_{t} u\Vert_{L^\infty_t H^2(\mathbb{R}^3)}+\Vert A\Vert_{L^\infty_t H^\frac{5}{2}(\mathbb{R}^3)}\Big)
\end{displaymath} 
Putting all together
\begin{align*}
\Vert\partial_{t} u\Vert_{L^\infty_t H^2(\mathbb{R}^3)}\lesssim\Vert\partial_{tt}u(0)\Vert_{L^2(\mathbb{R}^3)}+\Vert(A_0,A_1)\Vert_{H^\frac{5}{2}(\mathbb{R}^3)\times H^\frac{3}{2}(\mathbb{R}^3),}
\end{align*}
moreover one has
\begin{displaymath}
\Vert\partial_ {tt}u(0)\Vert_{L^2(\mathbb{R}^3)}\lesssim\Vert u_0\Vert_{H^4(\mathbb{R}^3)}+C(E,R)\Vert A_0\Vert_{H^\frac{5}{2}(\mathbb{R}^3),}
\end{displaymath}
then we get
\begin{equation}\label{eq:ok}
\Vert\partial_{t} u\Vert_{L^\infty_t H^2(\mathbb{R}^3)}\lesssim\Vert u_0\Vert_{H^4(\mathbb{R}^3)}+\Vert (A_0,A_1)\Vert_{H^\frac{5}{2}(\mathbb{R}^3)\times H^\frac{3}{2}(\mathbb{R}^3)}
\end{equation}


\begin{thebibliography}{100}
\bibitem{AM1}  P. Antonelli and P. Marcati, \emph{On the finite energy weak solutions to a system in Quantum Fluid Dynamics}, Comm. Math. Phys. {\bf 287} (2009), no 2, 657--686.
\bibitem{AM2} P. Antonelli and P. Marcati, \emph{The Quantum Hydrodynamics system in two space dimensions}, Arch. Rat. Mech. Anal. \textbf{203} (2012), 499--527.
\bibitem{AMContMath} P. Antonelli, P. Marcati, \emph{Some results on systems for quantum fluids}, accepted Cont. Math.
\bibitem{AIM} D. Ars\'enio, S. Ibrahim, N. Masmoudi, \emph{Derivation of the magnetohydrodynamic system from Navier-Stokes-Maxwell systems}, Arch. Rat. Mech. Anal. {\bf 216}, no. 3 (2015), 767--812.
\bibitem{AnCS} P. Antonelli,   R. Carles,  J.D. Silva,   \emph{Scattering for Nonlinear Schrodinger Equation Under Partial Harmonic Confinement} Comm.Math.Phys. {\bf 334} (2015), no.1, 367--396.
\bibitem{BT} I. Bejenaru, D. Tataru, \emph{Global well-posedness in the energy space for the Maxwell-Schr\"odinger system}, Comm. Math. Phys. {\bf 288} (2009), no. 1, 145--198.
\bibitem{DI} G. Da Prato, M. Iannelli, \emph{On a method for studying abstract evolution equations in the hyperbolic case}, Comm. PDEs {\bf 1} (1976), no. 6, 585--608.
\bibitem{Fe} R.P. Feynman, R.B. Leighton and M. Sands,  The Schr\"odinger equation in a classical context: a seminar on superconductivity (Chapter 21) \emph{The Feynman lectures on physics,}Vol. III \emph{Quantum Mechanics}, Addison-Wesley Publishing Co., Inc., Reading, Mass.-London, 1995.
\bibitem{GaM} I. Gasser,  P.A. Markowich, \emph{ Quantum hydrodynamics, Wigner transforms and the classical limit} Asymptot. Anal. {\bf 14} (1997), no. 2, 97--116. 
\bibitem{GV} J. Ginibre, G. Velo, \emph{Generalized Strichartz inequalities for the wave equation}, J. Funct. Anal. {\bf 133}, no. 1 (1995), 50--68.
\bibitem{GV1} J. Ginibre, G. Velo, \emph{Long range scattering and modified wave operators for the Maxwell-Schr\"odinger system I. The case of vanishing asymptotic magnetic field}, Comm. Math. Phys, {\bf 236} (2003), no. 3, 395--448.
\bibitem{GV2} J. Ginibre, G. Velo, \emph{Long range scattering and modified wave operators for the Maxwell-Schr\"odinger system II. The general case}, Ann. Henri Poincar\'e {\bf 8} (2007). no. 5, 917--994.
\bibitem{GV3} J. Ginibre, G. Velo, \emph{Uniqueness at infinity in time for the Maxwell-Schr\"odinger system with arbitrarily large asymptotic data}, Port. Math. {\bf 65} (2004), no. 4, 509--534.
\bibitem{GV4} J. Ginibre, G. Velo, \emph{Long range scattering for the Maxwell-Schr\"odinger system with arbitrarily large asymptotic data}, Hokkaido Math. J. {\bf 37} (2008), no. 4, 795--811.
\bibitem{GNS} Y. Guo, K. Nakamitsu, W. Strauss, \emph{Global finite-energy solutions to the Maxwell-Schr\"odinger system}, Comm. Math. Phys. {\bf 170} (1995), 181--196.
\bibitem{Haas} F. Haas, \emph{A magnetohydrodynamic model for quantum plasmas}, Phys. Plasmas {\bf 12} (2005), 062117.
\bibitem{HaasB} F. Haas, \emph{Quantum plasmas: An hydrodynamic approach}, New York: Springer.
\bibitem{Jun} J. Kato, \emph{Existence and uniqueness of the solution to the modified Schr\"odinger map}, Math. Res. Lett {\bf 12} (2005), no. 2-3, 171--186.
\bibitem{Kato1} T. Kato, \emph{Linear evolution equations of ``hyperbolic'' type}, J. Fac. Sci. Univ. Tokyo Sect. I {\bf 17} (1970), 241--258.
\bibitem{Kato2} T. Kato, \emph{Linear evolution equations of ``hyperbolic'' type, II}, J> Math. Soc. Japan {\bf 25} (1973), 648--666.
\bibitem{KT} M. Keel, T. Tao, \emph{Endpoint Strichartz Estimates}. Amer. J. Math. \textbf{120} (1998), 955--980.
\bibitem{KG} M.A. Kon, A. Gulisashvili, \emph{Exact smoothing properties of Schr\"odinger semigroups}, Amer. J. Math. {\bf 188}, no. 6 (1996), 1215--1248.
\bibitem{LL} E. Lieb, M. Loss, \textit{Analysis}, Graduate Studies in Mathematics, vol. 14, AMS, 2001.
\bibitem{Mu1} F. Murat, \emph{Compacit\'e par compensation}, Ann. Scuola Norm. Sup. Pisa Cl. Sci. (4) {\bf 5}, no. 3 (1978), 489--507.
\bibitem{Mu2} F. Murat, \emph{Compacit\'e par compensation: condition n\'ecessaire et suffisante de continuit\'e faible sous une hypoth\`ese de rang constant}, Ann. Scuola Norm. Sup. Pisa Cl. Sci. (4) {\bf 8}, no. 1 (1981), 69--102.
\bibitem{NT} K. Nakamitsu, M. Tsutsumi, \emph{The Cauchy problem for the coupled Maxwell-Schr\"odinger equations}, J. Math. Phys. {\bf 27} (1986), 211--216.
\bibitem{NW} M. Nakamura, T. Wada, \emph{Local well-posedness for the Maxwell-Schr\"odinger equation}, Math. Ann. {\bf 332} (2005), no. 3, 565--604.
\bibitem{NW1} M. Nakamura, T. Wada, \emph{Global existence and uniqueness of solutions to the Maxwell-Schr\"odinger equations}, Comm. Math. Phys. {\bf 276} (2007), 315--339.
\bibitem{A} A. Pazy, \emph{Semigroups of linear operators and applications to partial differential equations}, Appl. Math. Sciences, {\bf 44}, Springer-Verlag, New York, 1983.
\bibitem{Pet} K. Petersen, \emph{Existence of a unique local body solutions to the many-body Maxwell-Schr\"odinger initial value problem}, preprint archived as {arXiv:1402.3680v1}.
\bibitem{RS} M. Reed, B. Simon, \emph{Methods of modern mathematical physics II: Fourier analysis, self-adjointness}, Elsevier.
\bibitem{Sc} L.I. Schiff, \emph{Quantum Mechanics}, 2nd ed. New-York: McGraw-Hill, 1955.
\bibitem{Shi} A. Shimoura, \emph{Modified wave operators for Maxwell-Schr\"odinger equations in three space dimensions}, Ann. Henri Poincar\'e {\bf 4} (2003), 661--683.
\bibitem{SE} P.K. Shukla, B. Eliasson, \emph{Nonlinear aspects of quantum plasma physics}, Phys. Usp. {\bf 53} (2010), 51--76.
\bibitem{ShuEl} P.K. Shukla, B. Eliasson, \emph{Novel attractive force between ions in quantum plasmas}, Phys. Rev. Lett. {\bf 108} (2012), 165007.
\bibitem{Tao} T. Tao, \emph{Nonlinear dispersive equations: local and global analysis}, CBMS regional conference series in mathematics, 2006.
\bibitem{Tar} L. Tartar, Compensated compactness and applications to partial differential equations. \emph{Nonlinear analysis and mechanics: Heriot-Watt Symposium}, Vol. IV. \emph{Research Notes in Mathematics}, vol. 39, Pitman, Boston, Mass.-London, 136--212, 1979.
\bibitem{Tar1} L.Tartar, \emph{ An introduction to {N}avier-{S}tokes equation and oceanography} Lecture Notes of the Unione Matematica Italiana, vol. 1, Springer-Verlag, Berlin; UMI, Bologna, 2006
\bibitem{Tsu} Y. Tsutsumi, \emph{Global existence and asymptotic behavior of solutions for the Maxwell-Schr\"odinger equations in three space dimensions}, Comm. Math. Phys. {\bf 151} (1993), 543--576.
\bibitem{TN} M. Tsutsumi, K. Nakamitsu, \emph{Global existence of solutions to the Cauchy problem for coupled Maxwell-Schr\"odinger equations in two space dimensions}, in: J.H. Lightbourne, S.M. Rankin (eds.) Physical mathematics and nonlinear partial differential equations, New York: Marcel Dekker (1985), 139--155.
\bibitem{Wada} T. Wada, \emph{Smoothing effects for Schr\"odinger equations with electro-magnetic potentials and applications to the Maxwell-Schr\"odinger equations}, J. Funct. Anal. {\bf 263} (2012), 1--24.
\end{thebibliography}
          \end{document}